\documentclass[12pt, a4paper]{article}

\usepackage[a4paper, margin=3cm]{geometry}
\usepackage{amsmath, amssymb, amsthm, graphicx}

\allowdisplaybreaks[3]

\newtheorem{lem}{Lemma}[section]
\newtheorem{thm}[lem]{Theorem}

\newtheorem{conj}[lem]{Conjecture}

\numberwithin{equation}{section}
\numberwithin{figure}{section}

\renewcommand{\leq}{\leqslant}
\renewcommand{\geq}{\geqslant}

\newcommand{\dimB}{\dim_{\rm B}}
\newcommand{\dimH}{\dim_{\rm H}}

\newcommand{\bA}{\mathbf{A}}
\newcommand{\bc}{\mathbf{c}}
\newcommand{\be}{\mathbf{e}}
\newcommand{\Bf}{\mathbf{f}}
\newcommand{\bp}{\mathbf{p}}

\title{Some new classes of directed graph IFSs}     
\author{Graeme Boore }

\begin{document}

\maketitle

\begin{abstract}
It has been shown, see \cite{Paper_GCB_KJF}, that certain $2$-vertex directed graph iterated function systems (IFSs), defined on the unit interval  and satisfying the convex strong separation condition (CSSC), have attractors with components that are not standard IFS attractors where the standard IFSs may be with or without separation conditions. The proof required the multiplicative rational independence of parameters and the calculation of Hausdorff measure. In this paper we present a proof which does not have either of these requirements and so we identify a whole new class of $2$-vertex directed graph IFSs. 

We extend this result to $n$-vertex ($n\geq 2$, CSSC) directed graph IFSs, defined on the unit interval, with no effective restriction on the form of the associated directed graph, subject only to a condition regarding level-$1$ gap lengths. 

We also obtain a second result for $n$-vertex ($n\geq 2$, CSSC) directed graph IFSs, defined on the unit interval, which does require the calculation of Hausdorff measure.
\end{abstract}

\section{Introduction}\label{one}
In this paper we take some first steps towards a classification of (self-similar) directed graph IFSs defined on the unit interval. Directed graph IFSs, also known as graph directed IFSs, are one of the largest classes of deterministic IFSs. Probabilistic IFSs are often constructed by introducing random elements into the machinery of deterministic IFSs, see for example \cite{Book_Barnsley2}, so these results should be of interest across the whole range of IFS theory.

The Cantor set, the Sierpi\'nski triangle and the Menger sponge are well known examples of fractal sets  and each of them can be defined as the attractor of a (self-similar) $1$-vertex directed graph IFS, where the defining contractions are contracting similarities. For this reason it is convenient to use \emph{standard IFS} to mean a (self-similar) $1$-vertex directed graph IFS. Also we will often write \emph{$n$-vertex IFS} as a shortening of (self-similar) $n$-vertex directed graph IFS. 
 
Any $n$-vertex IFS determines a unique ($n$-component) attractor under the terms of the Contraction Mapping Theorem so it is natural to try and classify them by distinguishing between the components of their attractors. This we do in Theorems \ref{P2M} and \ref{P2Q} where we identify many new families of $n$-vertex ($n\geq 2$, CSSC) IFSs, defined on the unit interval, whose attractors have components that are not standard IFS attractors, where the standard IFS may be with or without separation conditions, overlapping or otherwise.  These results considerably extend those of \cite{Paper_GCB_KJF} where it was established that standard IFSs are in fact a proper subclass of  directed graph IFSs. Overall the work of this paper means that we can expect most $n$-vertex ($n\geq 2$, CSSC) IFSs to have attractors with components that are not standard IFS attractors. In particular the categorical nature of Theorem \ref{P2M} suggests that a precise classification is a real possibility. We discuss ways in which further progress may be made in Section \ref{six}. 

These results are also of interest because they provide information about properties that standard overlapping IFS attractors don't have. This is useful because apart from results about their Hausdorff dimension, see for example \cite{Paper_KJF4}, little is known about the structure of the attractors of overlapping IFSs in general. 
%trim=left bottom right top (also need clip).
\begin{figure}[htb]
\begin{center}
\includegraphics[trim = 10mm 185mm 10mm 15mm, clip, scale =0.7]{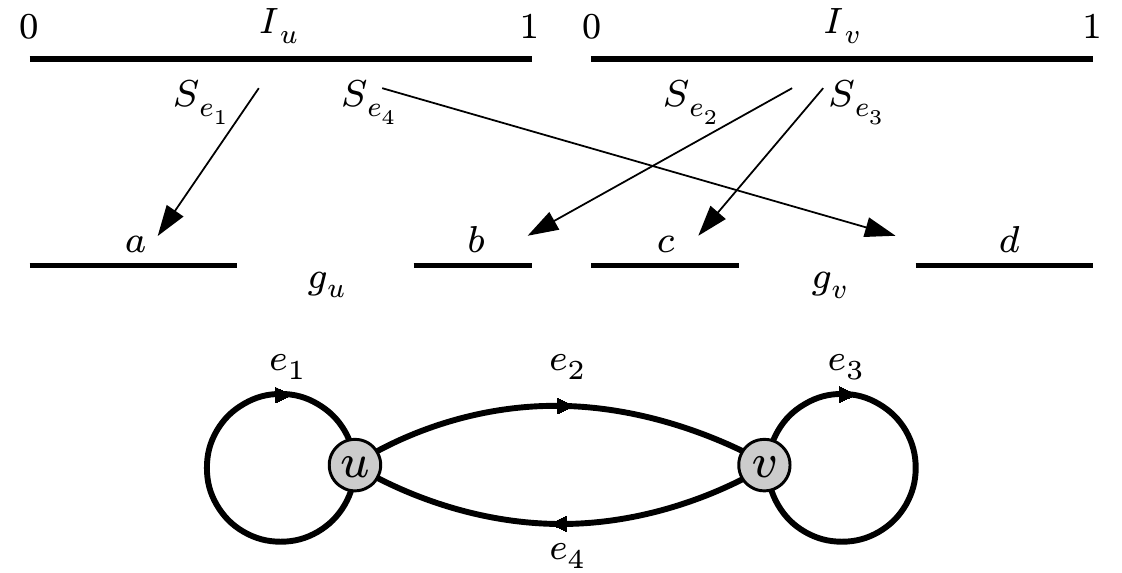}
\end{center}
\caption{A class of $2$-vertex directed graph IFSs defined on the unit interval, the similarities $S_{e_1},\, S_{e_2},\, S_{e_3}$ and $S_{e_4}$ do not reflect  and $\left\{g_u,g_v,a,b,c,d\right\}\subset \mathbb{R^+}$.}
\label{P22VertexUnitInterval}
\end{figure}

Theorem \ref{P2M} applies to one of the simplest types of $2$-vertex IFSs that can be defined on the unit interval as illustrated in Figure \ref{P22VertexUnitInterval}. The attractor of these systems consists of two non-empty compact sets, one at each vertex, which we write as $(F_u, F_v)$. Two of the results we proved in \cite{phdthesis_Boore,Paper_GCB_KJF} about these $2$-vertex IFS attractors are as follows, see \cite[Theorem 3.4.7, Theorem 3.5.8]{phdthesis_Boore}  or \cite[Theorem 4.6, Theorem 7.4]{Paper_GCB_KJF}. 
	
\begin{thm}
\label{2GthmN}
For the $2$-vertex IFS of Figure \ref{P22VertexUnitInterval} with attractor $(F_u, F_v)$ and $s=\dimH F_u$ $=\dimH F_v$, suppose that the following conditions hold
\begin{equation*}
\textup{(1)}  \quad \frac{1-a^s}{b^s} \leq 1, \quad \quad  \textup{(2)}  \quad \frac{(1-b)(1-a^s)}{ba^s}\geq 1.
\end{equation*}
Then
\begin{equation*}
\mathcal{H}^s(F_u) = 1 \quad \textrm{and} \quad \mathcal{H}^s(F_v)=\frac{1-a^s}{b^s}.
\end{equation*}
\end{thm}

\begin{thm}
\label{2GthmU}
For the $2$-vertex IFS of Figure \ref{P22VertexUnitInterval} with attractor $(F_u, F_v)$ suppose that conditions (1) and (2) of  Theorem \ref{2GthmN} hold, so that $\mathcal{H}^s(F_u)=1$, and suppose also that the set $\left\{g_u,g_v,a,b,c,d\right\}\subset \mathbb{R^+}$ is multiplicatively rationally independent. 

Then $F_u$ is not the attractor of any standard IFS, defined on $\mathbb{R}$, with or without separation conditions.
\end{thm}
We note that the standard IFSs in the statement of Theorem \ref{2GthmU} may have defining similarities which reflect. Whilst it should be clear from their proofs, we point out here that all the IFSs referred to in the statements of Theorems \ref{P2M}, \ref{P2Q} and \ref{P2T} that follow are assumed to have defining similarities that do not reflect and this is also the case for all the similarities depicted in diagrams in this paper. In fact it is not difficult to adjust Theorems \ref{P2Q} and \ref{P2T} so that they apply if reflecting similarities are allowed and we describe how this may be done in the comment after Theorem \ref{P2T} below. However it doesn't seem likely that this will be as easy to do for Theorem \ref{P2M}.

We prove Theorem \ref{P2M} in Section \ref{three}.  It removes the requirements of multiplicative rational independence of parameters and the calculation of Hausdorff measure from Theorem \ref{2GthmU}. It also  applies to both components and so identifies a whole new class of $2$-vertex IFS attractors. This means that whenever we apply Theorem \ref{2GthmN} we now know that we are calculating the Hausdorff measure of a different class of attractor. The condition $F_u\neq F_v$ is not in fact a restriction here for if  $F_u=F_v$ then the $2$-vertex IFS of Figure \ref{P22VertexUnitInterval} reduces to a standard IFS, see Lemmas \ref{P2G} and \ref{P2R}. In Section \ref{four} we present a specific example of an attractor $(F_u, F_v)$ where Theorem \ref{P2M} applies and for which we also calculate the Hausdorff measure.

\begin{thm}
\label{P2M}
For the attractor $(F_u, F_v)$ of the $2$-vertex IFS of Figure \ref{P22VertexUnitInterval} suppose that $F_u\neq F_v$.

Then neither $F_u$ nor $F_v$ is the attractor of any standard IFS, defined on $\mathbb{R}$, with or without separation conditions.
\end{thm}

The proof of Theorem \ref{P2Q} is given in Subsection \ref{five4}. Condition $(1)$ imposes no effective restriction on the form that the associated directed graph can take since if it doesn't hold then $F_u$ is a standard IFS attractor as we prove in Lemma \ref{P2nv1}. Condition (2) states that the maximum gap length of $F_u$ is not greater than some other specified level-$1$ gap lengths. Its purpose is to prevent similarity maps of $F_u$ from spanning the gaps between level-$1$ intervals (see Lemma \ref{P2nospan}). We discuss the possibility of weakening Condition (2) in Section \ref{six}. In Subsection \ref{five3} we show that $F_u \not\subset F_v$ in Condition (3) is in fact the correct generalisation of the condition $F_u \neq F_v$ of Theorem \ref{P2M}, and we also show that Theorem \ref{P2Q} is easy to apply in practice. Definitions of the terminology and notation used in the statement of this theorem are given in Section \ref{two}. 

\begin{thm}
\label{P2Q}
Let $\bigl( V, E^*, i, t, r, ((\mathbb{R},\left| \ \  \right|))_{v \in V}, (S_e)_{e \in E^1} \bigr)$ be an $n$-vertex ($n\geq 2$, CSSC) IFS, defined on the unit interval, with attractor $(F_u)_{u \in V}$. Let $u \in V$ be fixed and suppose the following conditions hold. 

\medskip

\textup{(1)}  There is some $w \in V$, $w\neq u$, and a simple cycle $\bc_w$ attached to $w$ but not to $u$.

 Let $\bp \in E_{uw}^*$ be a simple path and let $V^\prime$ denote the set of all vertices in the vertex lists of $\bp$ and $\bc_w$. 

\textup{(2)} For each $v \in V^\prime$, $\max G_u \leq \min G_v^1$.

\textup{(3)} For each $v \in V^\prime \setminus \left\{u\right\}$,  $F_u \not\subset F_v$.

\medskip

Then $F_u$ is not the attractor of any standard IFS, defined on $\mathbb{R}$, with or without separation conditions.

\end{thm}

The new class of $2$-vertex IFSs identified in Theorem \ref{P2M} is extended considerably by Theorem \ref{P2Q}. As an example, suppose we add any number of edges (with non-reflecting similarities) to the directed graph of Figure \ref{P22VertexUnitInterval} maintaining the CSSC. Suppose also that all level-$1$ gap lengths are kept equal across both vertices (see Subsection \ref{five2}, Equations \eqref{G_uG_u^1}) and that $F_u \not\subset F_v$ and $F_v \not\subset F_u$, then Theorem \ref{P2Q} applies and neither $F_u$ nor $F_v$ is a standard IFS attractor. 

Finally we state a theorem which is of theoretical interest, although for practical purposes it is superceded by Theorem \ref{P2M}. This is because there are not many self-similar sets for which the Hausdorff measure is known and at the time of writing the only components of  $n$-vertex ($n\geq 2$, CSSC) IFS attractors that we can calculate the Hausdorff measure for are the components of the attractors of $2$-vertex IFSs of the type shown in Figure \ref{P22VertexUnitInterval}. For the Hausdorff measure of the attractors of standard (COSC) IFSs defined on the unit interval, see \cite{Paper_Ayer_Strichartz} and \cite{Paper_Marion}. In fact (as far as I'm aware) the exact Hausdorff measure hasn't been calculated for any self-similar set of non-integral Hausdorff dimension greater than $1$, see \cite{Paper_ZhouFeng}. We outline a proof of Theorem \ref{P2T} in Subsection \ref{five5}.
\begin{thm}
\label{P2T}
Let $\bigl( V, E^*, i, t, r, ((\mathbb{R},\left| \ \  \right|))_{v \in V}, (S_e)_{e \in E^1} \bigr)$ be an $n$-vertex ($n\geq 2$, CSSC) IFS, defined on the unit interval, with attractor $(F_u)_{u \in V}$ and $s=\dimH F_u$. Suppose that the number of edges in the directed graph is minimal. Let $u \in V$ be fixed and suppose the following conditions hold. 

\medskip

\textup{(1)}  There is some $w \in V$, $w\neq u$, and a simple cycle $\bc_w$ attached to $w$ but not to $u$.

 Let $\bp \in E_{uw}^*$ be a simple path and let $V^\prime$ denote the set of all vertices in the vertex lists of $\bp$ and $\bc_w$. 

\textup{(2)} $\mathcal{H}^s(F_{u})=1$ and, for each $v \in V^\prime$ and all $e \in E_v^1$, $\mathcal{H}^s(F_{t(e)})=1$.

\textup{(3)} For each $v \in V^\prime \setminus \left\{u\right\}$,  $F_u \not\subset F_v$.

\medskip

Then $F_u$ is not the attractor of any standard IFS, defined on $\mathbb{R}$, with or without separation conditions.

\end{thm}

We can modify Theorems \ref{P2Q} and \ref{P2T} so that they also hold if we allow the $n$-vertex and standard IFSs in their statements to be defined by similarities which may reflect. All we need to do is to change Condition (3) so that it reads
\begin{center}
(3) \emph{For each $v \in V^\prime \setminus \left\{u\right\}$,  $F_u \not\subset F_v$ and $F_u \not\subset R(F_v)$},
\end{center}
where $R$ is the function $R(x)=1-x$ which reflects about $x=1/2$. The proof in Subsection \ref{five4} can now be adjusted accordingly since Lemmas \ref{P2nospan} and \ref{2GlemU} also hold for reflecting similarities. 

\section{Notation and background theory}\label{two}
We often use a notation of the form $(A_c)_{c\in B}$ and $(A)_{c\in B}$ when $B$ is a finite set of $n$ elements as this is just a convenient way of writing down ordered $n$-tuples. That is, if $B$ is ordered as $B=(b_1,b_2,\ldots,b_n)$, then $(A_c)_{c\in B}=(A_{b_1},A_{b_2},\ldots,A_{b_n})$ and $(A)_{c\in B}=(A,A,\ldots,A)$.

Apart from mappings of the form $S^k$, we will use $\circ$ for the composition of mappings throughout. The order of composition is $(S\circ T)(x)=S(T(x))$.

For further background theory, definitions and references to other source material see \cite{phdthesis_Boore}. 

\subsection{Directed graph IFSs}\label{two1} 

We use $\bigl(V,E^*,i,t,r,((X_{v},d_{v}))_{v \in V},(S_e)_{e \in E^1}\bigr)$ to indicate a \emph{directed graph IFS} where $\bigl(V,E^*,i,t\bigr)$ is the associated \emph{directed graph}, $V$ is the set of all vertices, $E^*$ is the set of all finite (directed) paths, $i:E^* \to V$ and $t:E^* \to V$ are the initial and terminal vertex functions. The set of all (directed) edges in the graph, that is the set of paths of length $1$, is written as $E^1$ with $E^1\subset E^*$. $V$ and $E^1$ are always assumed to be finite sets. We use $E^1_u$ to indicate the set of all edges leaving the vertex $u$, $E^k_u$ for the set of all paths of length $k$ leaving the vertex $u$, $E^k_{uv}$ for the set of all paths of length $k$ starting at the vertex $u$ and finishing at $v$ and so on. 

A finite \emph{(directed) path} $\be \in E^*$ is a finite string of consecutive edges so a path of length $k$ can be written as $\be=e_1 \cdots e_k$ for some edges $e_i \in E^1$ with $t(e_i)=i(e_{i+1})$ for $1\leq i < k$. The initial vertex of a path is the initial vertex of its first edge so $i(\be)=i(e_1)$ and similarly $t(\be)=t(e_k)$. A \emph{cycle} is a path with the same initial and terminal vertices. A \emph{loop} is an edge $e\in E^1$ with $i(e)=t(e)$, that is a cycle of length $1$. The \emph{vertex list} of a path $\be=e_1\cdots e_k \in E^*$ is  $v_1v_2v_3\cdots v_{k+1} = i(e_1) t(e_1) t(e_2) \cdots $ $t(e_k)$ and shows the order in which a path visits its vertices. A \emph{simple path}, which is not a cycle, visits no vertex more than once so a path $\bp=e_1\cdots e_k\in E^*$ is simple if its vertex list contains exactly $k+1$ different vertices. A \emph{simple cycle} is a cycle which visits no vertex more than once apart from the initial and terminal vertices which are the same so a cycle $\bc=e_1\cdots e_k\in E^*$ is simple if its vertex list contains exactly $k$ different vertices. We say that \emph{two distinct paths are attached} if their vertex lists contain a common vertex. A \emph{path $\be$ is attached to a vertex $v$} if $v$ is in the vertex list of $\be$. 

We assume the directed graph is strongly connected and that each vertex in the directed graph has at least two edges leaving it, this is to avoid components of the attractor (defined below) that consist of single point sets or are just scalar copies of those at other vertices (see \cite{Paper_Edgar_Mauldin}). The \emph{contraction ratio function} $r:E^*\to (0,1)$ assigns contraction ratios to the finite paths in the graph. To each vertex $v \in V$ is associated the non-empty complete metric space $(X_{v},d_{v})$ and to each directed edge $e\in E^1$ is assigned a contraction $S_{e}:X_{t(e)} \to X_{i(e)}$ which has the contraction ratio given by the function $r(e) = r_{e}$. We follow the convention already established in the literature, see \cite{Book_Edgar2, Paper_Edgar_Mauldin}, that $S_e$ maps in the opposite direction to the direction of the edge $e$ that it is associated with in the graph.  The contraction ratio along a path $\be=e_1e_2\cdots e_k \in E^*$ is defined as $r(\be)=r_{\be}=r_{e_1}r_{e_2}\cdots r_{e_k}$. The ratio $r_{\be}$ is the ratio for the contraction $S_{\be}:X_{t(\be)} \to X_{i(\be)}$ along the path $\be$ where $S_{\be}=S_{e_1}\circ S_{e_2}\circ \cdots \circ S_{e_k}$.

In this paper we are only going to be concerned with $n$-vertex IFSs defined on the unit interval (see below) where $((X_{v},d_{v}))_{v \in V} = ((\mathbb{R}^m,\left| \ \  \right|))_{v \in V}$ with $m=1$ and $(S_e)_{e \in E^1}$ are contracting similarities and not just contractions, however we give the remaining definitions and background results for general $n$-vertex IFSs defined on $m$-dimensional Euclidean space. We use $K(\mathbb{R}^m)$ to denote the set of all non-empty compact subsets of $\mathbb{R}^m$. Using the Contraction Mapping Theorem it can be shown that an $n$-vertex IFS $\bigl( V, E^*, i, t, r, ((\mathbb{R}^m,\left| \ \  \right|))_{v \in V}, (S_e)_{e \in E^1} \bigr)$ determines a unique list of non-empty compact sets $(F_u)_{u \in V} \in (K(\mathbb{R}^m))^n$ which satisfies the invariance equation 
\begin{equation}
\label{Invariance equation3}
 (F_u)_{u \in V} = \biggl( \ \bigcup_{ \substack{e\in E_u^1} }S_e(F_{t(e)}) \ \biggr)_{u \in V},
\end{equation} 
see  \cite[Theorem 1.3.4]{phdthesis_Boore}, \cite[Theorem 4.3.5]{Book_Edgar2} or \cite[Theorem 1]{Paper_Mauldin_Williams}. Under the terms of the Contraction Mapping Theorem $(F_u)_{u \in V}$ is known as the \emph{attractor} of the system and  we call the $n$ non-empty compact sets $F_u$,  $u \in V$, the \emph{components of the attractor}. 

The \emph{open set condition} (OSC) is satisfied if and only if there exist non-empty bounded open sets, $(U_u)_{u\in V}\subset (\mathbb{R}^m)^n$ such that for each $u\in V$
\begin{gather*}
S_e(U_{t(e)}) \subset U_u   \textrm{ for all } e \in E_u^1, \\
\textrm{ and }S_e(U_{t(e)})\cap S_f(U_{t(f)}) = \emptyset \textrm{ for all } e,f\in E_u^1,\textrm{ with } e \neq f.
\end{gather*}
See \cite{Paper_Hutchinson}, \cite{Book_KJF2} or \cite{Book_Edgar2}.

The \emph{convex open set condition} (COSC) is satisfied if and only if the OSC is satisfied for non-empty bounded open sets $(U_u)_{u\in V}\subset (\mathbb{R}^{n})^{\#V}$, where these sets are also convex. See \cite{Paper_Elkes_Keleti_Mathe,Paper_FengWang}. 

The \emph{strong separation condition} (SSC) is satisfied if and only if for each $u\in V$,
\begin{equation*}
S_e(F_{t(e)})\cap S_f(F_{t(f)}) = \emptyset \textrm{ for all } e,f\in E_u^1,\textrm{ with } e \neq f.
\end{equation*}

We write $C(F_u)$ for the convex hull of $F_u$. 

The \emph{convex strong separation condition} (CSSC) is satisfied if and only if for each $u\in V$,
\begin{equation*}
S_e(C(F_{t(e)}))\cap S_f(C(F_{t(f)})) = \emptyset \textrm{ for all } e,f\in E_u^1,\textrm{ with } e \neq f.
\end{equation*}
The CSSC holds for the $2$-vertex IFSs illustrated in Figures \ref{P22VertexUnitInterval}, \ref{Paper2Example}, \ref{P2VertexOneLoop}, \ref{P2VertexSubset} and \ref{CounterExample}.

For a set $A\subset \mathbb{R}^m$, we use the usual notation $\mathcal{H}^s(A)$ for the $s$-dimensional Hausdorff measure, $\dimH A$  and $\dimB A$ for the Hausdorff and box-counting dimension. The next theorem is the main dimension result for $n$-vertex IFSs defined on $\mathbb{R}^m$, see \cite[Theorem 6.9.6]{Book_Edgar2} or \cite[Theorem 3]{Paper_Mauldin_Williams}. For standard IFSs with $n=1$ this is the same as \cite[Theorem 9.3]{Book_KJF2}.

\begin{thm} 
\label{dimension}
Let $\bigl(V,E^*,i,t,r,((\mathbb{R}^m,\left| \ \  \right|))_{v \in V},(S_e)_{e \in E^1}\bigr)$ be an $n$-vertex IFS with attractor $(F_u)_{u \in V}$ where the mappings $(S_e)_{e \in E^1}$ are contracting similarities. Let $\bA(t)$ denote the $n\times n$ matrix 
whose $uv$th entry is
\begin{equation*}
A_{uv}(t) = \sum_{e\in E_{uv}^1} r_e^t,
\end{equation*}
let $\rho\left(\bA(t)\right)$ be the spectral radius of $\bA(t)$, and let $s$ be the unique non-negative real number that is the solution of $\rho\left(\bA(t)\right)=1$.  

If the OSC is satisfied then, for each $u \in V$, $s = \dimH F_u = \dimB F_u$ and $0 < \mathcal{H}^s \left(F_u\right) < +\infty$.
\end{thm} 

We say that an $n$-vertex IFS, $\bigl( V, E^*, i, t, r, ((\mathbb{R},\left| \ \  \right|))_{v \in V}, (S_e)_{e \in E^1} \bigr)$, is \emph{defined on the unit interval} if  its attractor, $(F_u)_{u \in V} \in (K(\mathbb{R}))^n$, is such that $\left\{0, 1\right\}\subset F_u \subset I_u=[0,1]$ for each $u \in V$. Here $I_u$ is the smallest convex set containing $F_u$ with $I_u=C(F_u)$. We use the notation $(I_u)_{u \in V}$ instead of $([0,1])_{u \in V}$ as it is useful for keeping track of the direction of similarities.  

In general, for any set $A\subset \mathbb{R}$, we call $[p,q]$ a \emph{gap interval} of $A$ if $\left\{p,q\right\} \subset A$ and $(p,q)\cap A=\emptyset$. A \emph{gap length} is the length of a gap interval and we use $G(A)$ to indicate the set of all gap lengths of $A$.  

For each $u \in V$, $F_u^k$ denotes the union of the \emph{level-$k$ intervals} of $F_u$ given by
\begin{equation*}
F_u^k=\bigcup_{ \substack{\be\in E_u^k} }S_{\be}(I_{t(\be)}),
\end{equation*}
where we put $F_u^0=I_u$. Some level-$k$ intervals are illustrated in Figure \ref{Paper2Example}. If $S_{\be}(I_{t(\be)})$ is a level-$k$ interval then $S_{\be}(F_{t(\be)})$ is the corresponding \emph{level-$k$ elementary piece}. In general $S_{\be}(F_{t(\be)})$ is an \emph{elementary piece} if $\be \in E^*$. The set of \emph{level-$k$ gap lengths} of $F_u$ is $G_u^k = \left\{ \left|J\right|    :   J \textrm{ is an open interval in } I_u \setminus F_u^k \right\} = G(F_u^k)$. The set of \emph{gap lengths} of $F_u$ is $G_u = \left\{ \left|J\right|    :   J \textrm{ is an open interval in } I_u \setminus F_u \right\}$ $= G(F_u)$. As $F_u=\bigcap_{k=0}^{\infty}F_u^k$ we also have $G_u=\bigcup_{k=1}^{\infty}G_u^k$. For $n$-vertex IFSs, defined on the unit interval, the  CSSC ensures that the sets of gap lengths and level-$k$ gap lengths of each component of the attractor exist and are well-defined, see \cite[Section 2.2]{phdthesis_Boore}.

\subsection{The $2$-vertex IFS of Figure \ref{P22VertexUnitInterval}}

In this subsection we give some basic information regarding the $2$-vertex (CSSC) IFS of Figure \ref{P22VertexUnitInterval}.

In all the figures of this paper lower case letters are used to indicate lengths of intervals and gap intervals, so the parameters $\left\{g_u,g_v,a,b,c,d\right\}\subset \mathbb{R}^+$  in Figure \ref{P22VertexUnitInterval} are such that
\begin{equation*} 
a+g_u+b=c+g_v+d=1.
\end{equation*}

The level-$1$ intervals of $F_u$  are
\begin{equation*}
S_{e_1}(I_u)=[0,a], \quad S_{e_2}(I_v)=[a+g_u,1], 
\end{equation*}
and the level-$1$ intervals of $F_v$ are
\begin{equation*}
S_{e_3}(I_v)=[0,c], \quad S_{e_4}(I_u)=[c+g_v,1].
\end{equation*}

The contracting similarity ratios of the similarities are
\begin{align*}
r_{e_1}=\frac{\left|S_{e_1}(I_u)\right|}{\left|I_u\right|}=a, \quad r_{e_2}=\frac{\left|S_{e_2}(I_v)\right|}{\left|I_v\right|}=b, \\
r_{e_3}=\frac{\left|S_{e_3}(I_v)\right|}{\left|I_v\right|}=c, \quad r_{e_4}=\frac{\left|S_{e_4}(I_u)\right|}{\left|I_u\right|}=d.
\end{align*}
and the similarities are    
\begin{align*}
&S_{e_1}(x)=r_{e_1}x=ax, \quad S_{e_2}(x)=r_{e_2}x+a+g_u=bx+a+g_u, \\
&S_{e_3}(x)=r_{e_3}x=cx, \quad S_{e_4}(x)=r_{e_4}x+c+g_v=dx+c+g_v.
\end{align*}

From the Invariance Equation \eqref{Invariance equation3}
\begin{align}
\label{Invariance equation1}
&F_u=S_{e_1}(F_u) \cup S_{e_2}(F_v), \\
\label{Invariance equation2}
&F_v=S_{e_3}(F_v) \cup S_{e_4}(F_u).
\end{align}

The sets of gap lengths $G_u$ and $G_v$ can be expressed as a finite union of cosets of finitely generated semigroups as  
\begin{gather}
\label{G_u}
G_u = g_u \left\langle 1,a \right\rangle \cup bdg_u\left\langle 1,a,bd,c\right\rangle  \cup  bg_v\left\langle 1,a,bd,c\right\rangle, \\
\label{G_v}
G_v = g_v \left\langle 1,c \right\rangle \cup bdg_v\left\langle 1,a,bd,c\right\rangle  \cup  dg_u\left\langle 1,a,bd,c\right\rangle,
\end{gather}
see \cite[Proposition 2.3.4 and Section 2.4]{phdthesis_Boore}. 

In the next lemma we collect together some basic results about maximum gap lengths that will be referred to later.
\begin{lem}
\label{P2I}
\begin{equation*}
\textup{(a)} \  \max G_u = \max\left\{g_u, bg_v\right\}, \quad \textup{(b)} \  \max G_v = \max\left\{g_v, dg_u\right\}.
\end{equation*}
Let $T: \mathbb{R} \rightarrow \mathbb{R}$ be a contracting similarity with contracting similarity ratio $r_T$, $0<r_T<1$. Then
\begin{equation*}
\textup{(c)} \  \max G(T(F_u)) = \max\left\{r_Tg_u, r_Tbg_v\right\}, \  \textup{(d)} \  \max G(T(F_v)) = \max\left\{r_Tg_v, r_Tdg_u\right\}.
\end{equation*}
\end{lem}
\begin{proof}Parts (a) and (b) follow immediately by Equations \eqref{G_u} and \eqref{G_v}. 

Clearly $G\left(T(F_u))\right)=r_TG_u$. This proves (c) and the proof of (d) is similar. 
\end{proof}

If $F_u=F_v$ then the $2$-vertex IFS of Figure \ref{P22VertexUnitInterval} reduces to a standard IFS with $F_u$ the attractor for $\left\{S_{e_1}, S_{e_2}\right\}$, see Lemma \ref{P2R}. It's easy to ascertain whether or not $F_u=F_v$ as the next lemma shows.
\begin{lem}
\label{P2G}
For the attractor $(F_u, F_v)$ of the $2$-vertex IFS of Figure \ref{P22VertexUnitInterval}
\begin{equation*}
F_u=F_v \ \ \text{if and only if} \ \ a=c \  \text{and} \  b=d.
\end{equation*}
\end{lem}
\begin{proof}(a) \emph{If $a=c$ and $b=d$ then $F_u=F_v$.} 

If $a=c$ and $b=d$ then $g_u=g_v$ with $S_{e_1}=S_{e_3}$ and $S_{e_2}=S_{e_4}$ which means the level-$k$ intervals must be the same at both vertices. That is $F_u^k =F_v^k$ for all $k \in \mathbb{N}\cup\left\{0\right\}$ and so
\begin{equation*}
F_u=\bigcap_{k=0}^\infty F_u^k=\bigcap_{k=0}^\infty F_v^k=F_v,
\end{equation*}
see \cite[Equation 2.2.2]{phdthesis_Boore}. 

\bigskip

(b) \emph{If $F_u=F_v$ then $a=c$ and $b=d$.} 

As $F_u=F_v$ implies $G_u=G_v$ it follows by Lemma \ref{P2I}(a) and (b) that 
\begin{equation*}
\max\left\{g_u, bg_v\right\}=\max G_u =\max G_v = \max\left\{g_v, dg_u\right\},
\end{equation*}
which is enough to prove $g_u=g_v$ and so $a+b=c+d$. 

If $g_u=g_v$ then $\max G_u=g_u>bg_v$ and $[a,a+g_u]$ is the unique maximum length gap interval of $F_u$. Similarly $[c,c+g_v]$ is the unique maximum length gap interval of $F_v$ so that $[a,a+g_u] = [c,c+g_v]$ with $a=c$ and $b=d$.
\end{proof}

We now introduce some convenient notation for a few useful gap intervals that will be needed in Section \ref{three}. Let
\begin{align*}
I_{g_u} &= [a,a+g_u], \\
I_{g_v} &= [c,c+g_v], \\
I_{bg_v} &= [1-bd-bg_v,1-bd], \\
I_{bc^kg_v} &= [a+g_u+bc^{k+1},a+g_u+bc^{k+1}+bc^kg_v], \\
I_{da^kg_u} &= [c+g_v+da^{k+1},c+g_v+da^{k+1}+da^kg_u], \\
I_{b^{i+1}d^{i}g_v} &= [1-b^{i+1}d^{i+1}-b^{i+1}d^{i}g_v,1-b^{i+1}d^{i+1}], \\
I_{b^{i+1}d^{i+1}g_u} &= [1-b^{i+2}d^{i+1}-b^{i+1}d^{i+1}g_u,1-b^{i+2}d^{i+1}].
\end{align*}
Each subscript is the interval length and each of these intervals is indicated by its subscript whenever it appears in a diagram.
  
\section{$F_u \neq F_v$ is sufficient}\label{three}

In this section we prove Theorem \ref{P2M}. First we show, in Lemmas \ref{P2O} and \ref{P2P}, that any (non-reflecting) similarities $S: F_u \rightarrow F_u$ or $S: F_u \rightarrow F_v$ are such that $S(I_u)$ can never span the gap between level-$1$ intervals. This property has a key role to play in the proof of Subsection \ref{three1} and throughout the rest of this paper.

The proofs of Lemmas \ref{P2O} and \ref{P2P} make use of the deceptively rich structure of these simple $2$-vertex systems. We illustrate the main steps in Figures \ref{induction1} and \ref{induction3}. Because $S(I_u)$ spans the gap between level-$1$ intervals, both of these figures are illustrating situations that turn out not to be possible. Also under the right circumstances each parameter can take any value in the range $(0,1)$. All this means that the lengths of image intervals depicted are not always believable but both figures do illustrate all the possible cases topologically and this is what is important for the proofs. 

We are able to prove directly that the various cases of (a) and (b) in Figures \ref{induction1} and \ref{induction3} can't occur. However a direct proof that case (c) can't happen doesn't seem to be possible. This is why we resort to proofs by induction in order to obtain the required contradictions.

\begin{lem}
\label{P2O}
For the $2$-vertex IFS of Figure \ref{P22VertexUnitInterval}, let $S: \mathbb{R} \rightarrow \mathbb{R}$ be a (non-reflecting) contracting similarity with contracting similarity ratio $r$, $0<r<1$, such that $S(F_u) \subset F_u$. 

Then either $S(I_u)\subset S_{e_1}(I_u)$ or $S(I_u)\subset S_{e_2}(I_v)$.
\end{lem}

\begin{figure}[ht]
\begin{center}
\includegraphics[trim = 14mm 160mm 15mm 15mm, clip, scale =0.7]{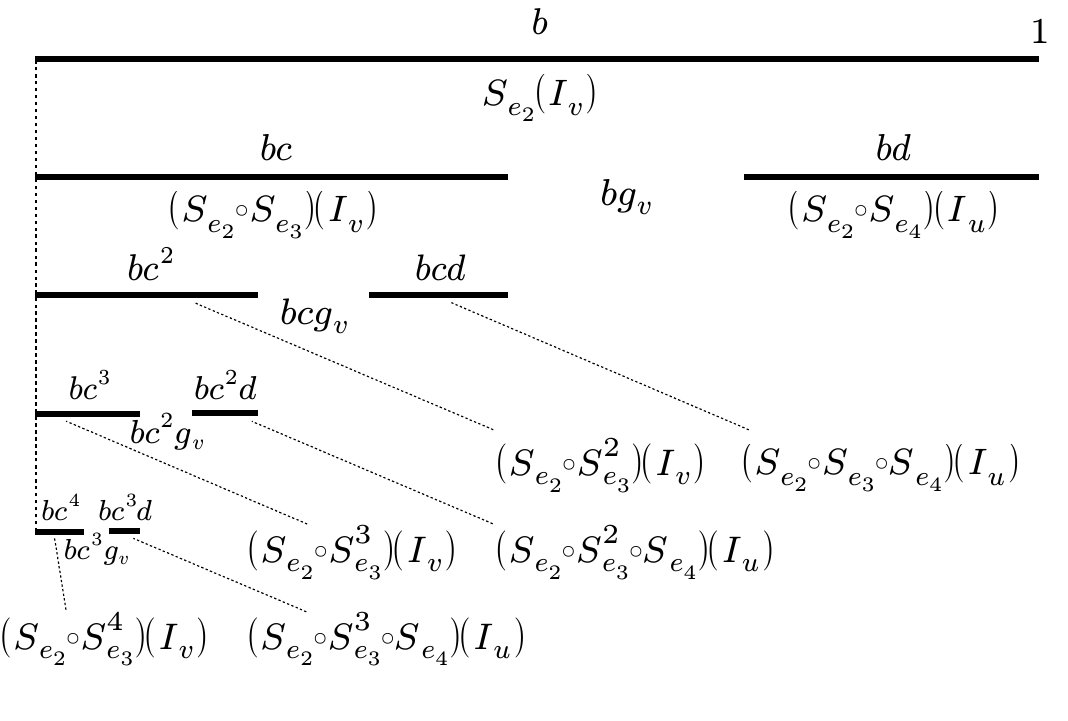}
\end{center}
\caption{The intervals $(S_{e_2}\circ S_{e_3}^{k+1})(I_v)$ and $(S_{e_2}\circ S_{e_3}^k\circ S_{e_4})(I_u)$ for $0\leq k \leq 3$. All the intervals shown are contained in $I_u$.}
\label{P2b}
\end{figure}

\begin{figure}[htb]
\begin{center}
\includegraphics[trim = 15mm 120mm 15mm 15mm, clip, scale =0.7]{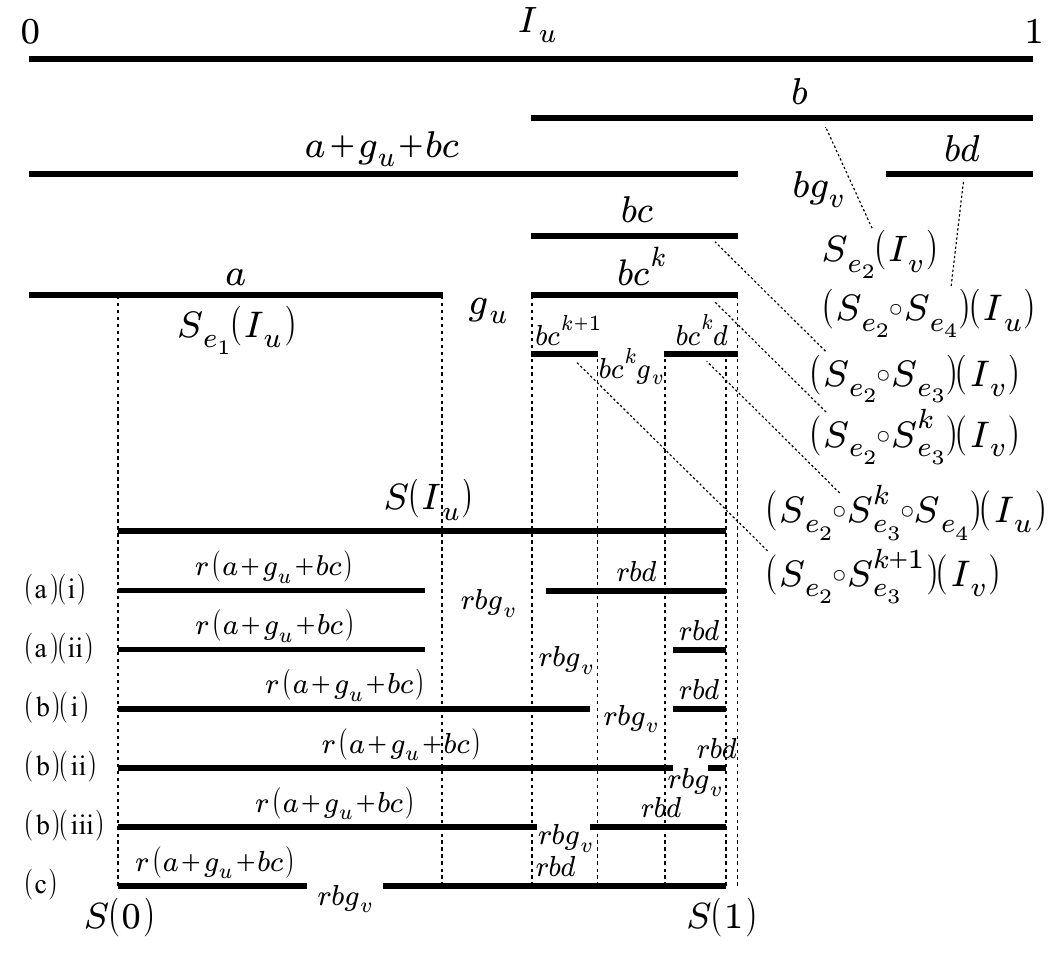}
\end{center}
\caption{The possible cases for Lemma \ref{P2O}. In the diagram $k=1$.}
\label{induction1}
\end{figure}

\begin{proof} For a contradiction we assume that $S(I_u)$ spans the gap between the level-$1$ intervals with $I_{g_u}\subset S(I_u)$, $S(0) \in S_{e_1}(I_u)$ and $S(1) \in S_{e_2}(I_v)$, as illustrated in Figure \ref{induction1}.

As $I_{g_u}\subset S(I_u)$ it follows that
\begin{align*}
g_u &\leq \max G(S(F_u)) \\
&= \max \left\{rg_u, rbg_v\right\} &&(\text{by Lemma \ref{P2I}(c)}), \\
&= rbg_v  &&(\text{as }rg_u < g_u),
\end{align*}
which proves 
\begin{equation}
\label{P2OA}
g_u < bg_v < g_v
\end{equation}
and ensures that $I_{bg_v}$ is the unique maximum length gap interval of $F_u$. 

It must be the case that $S(1) \in (S_{e_2}\circ S_{e_3})(I_v)$ as all gap intervals in  $S(F_u)$ are of shorter length than $I_{bg_v}$.  Also $S(1) > a+g_u$ since there are points of  $F_u$ as close as we like to $1$ on its left and there are no such points immediately to the left of  $a+g_u$. It follows that there exists a $k \in \mathbb{N}$ such that 
\begin{equation*}
a+g_u+bc^{k+1} = (S_{e_2}\circ S_{e_3}^{k+1})(1) < S(1) \leq (S_{e_2}\circ S_{e_3}^{k})(1) =  a+g_u+bc^{k}.
\end{equation*}
For such $k \in \mathbb{N}$
\begin{equation}
\label{S1}
S(1) \in (S_{e_2}\circ S_{e_3}^k\circ S_{e_4})(I_u)
\end{equation}
as should be clear from Figure \ref{P2b}. For the rest of this proof we consider $k$ in \eqref{S1} to be fixed. The example shown in Figure \ref{induction1} has $k=1$.

For $i\in \mathbb{N}$ let $P(i)$ be the statement
\begin{equation*}
I_{g_u}\subset (S\circ (S_{e_2}\circ S_{e_4})^i)(I_u).
\end{equation*}
The self-similar properties of the system permit a proof by induction that $P(i)$ holds for all $i\in \mathbb{N}$, and this is enough for a contradiction as it forces $g_u=0$.

\medskip

\emph{Induction base}. \\
\indent There are just three possible destinations for $I_{bg_v}$ under $S$
\begin{align*}
&\textup{(a)} \quad I_{g_u}\subset S(I_{bg_v}), \\
&\textup{(b)} \quad S(I_{bg_v})\subset  (S_{e_2}\circ S_{e_3}^{k})(I_v), \\
&\textup{(c)} \quad S(I_{bg_v})\subset  S_{e_1}(I_u).
\end{align*}
We now prove that neither (a) nor (b) can happen which leaves (c). This is enough to prove $P(1)$ as should be clear from Figure \ref{induction1}(c). We consider all the possible cases that can arise for (a) and (b) as follows. 

(a)(i) $ I_{g_u}\subset S(I_{bg_v})$ \emph{and} $I_{bc^kg_v} \subset (S\circ S_{e_2}\circ S_{e_4})(I_u) \subset (S_{e_2}\circ S_{e_3}^{k})(I_v)$.

This is the situation shown in Figure \ref{induction1}(a)(i). Here $rbd \leq bc^k$ and it must be the case that
\begin{align*}
bc^kg_v &\leq \max G((S\circ S_{e_2}\circ S_{e_4})(F_u)) \\
&= \max \left\{rbdg_u, rb^2dg_v\right\} &&(\text{by Lemma \ref{P2I}(c)}) \\
&= rb^2dg_v  &&(\text{by } \eqref{P2OA}) \\
&\leq b^2c^kg_v
\end{align*}
which is a contradiction.

(a)(ii) $ I_{g_u}\subset S(I_{bg_v})$ \emph{and} $(S\circ S_{e_2}\circ S_{e_4})(I_u) \subset (S_{e_2}\circ S_{e_3}^{k}\circ S_{e_4})(I_u)$.

Clearly $rbd \leq bc^kd$ so that $r \leq c^k$. As illustrated in Figure \ref{induction1}(a)(ii)
\begin{equation*}
I_{g_u}\cup  (S_{e_2}\circ S_{e_3}^{k+1})(I_v) \cup I_{bc^kg_v} \subset S(I_{bg_v})
\end{equation*}
which implies 
\begin{equation*}
g_u + bc^{k+1} + bc^kg_v \leq rbg_v \leq bc^kg_v
\end{equation*}
and is another contradiction.

(b)(i) $S(I_{bg_v})\subset  (S_{e_2}\circ S_{e_3}^{k})(I_v)$, $(S \circ S_{e_2} \circ  S_{e_4})(I_u) \subset  (S_{e_2}\circ S_{e_3}^{k}\circ S_{e_4})(I_u)$  \emph{and} $I_{bc^kg_v} \subset S(I_{bg_v})$.

In this case $bc^kg_v \leq rbg_v$ and $rbd \leq bc^kd$ which implies $r=c^k$. This forces the following equalities
\begin{align*}
(S \circ S_{e_2} \circ  S_{e_4})(I_u) &= (S_{e_2}\circ S_{e_3}^{k}\circ S_{e_4})(I_u) \\
S(I_{bg_v}) &= I_{bc^kg_v} \\
 (S \circ S_{e_2} \circ S_{e_3})(I_v) &=  (S_{e_2}\circ S_{e_3}^{k+1})(I_v)
\end{align*}
which means that $S \circ S_{e_2} = S_{e_2}\circ S_{e_3}^{k}$. However immediately to the left of  $S_{e_2}(I_v)$ is the gap interval $I_{g_u}$ and immediately to the left of $ (S_{e_2}\circ S_{e_3}^{k})(I_v)$ there is also the gap interval $I_{g_u}$  as shown in Figure \ref{induction1}(b)(i). Therefore
\begin{equation*}
I_{g_u} \subset S(I_{g_u})
\end{equation*}
which is impossible.

(b)(ii) $S(I_{bg_v})\subset  (S_{e_2}\circ S_{e_3}^{k})(I_v)$ \emph{and} $S(I_{bg_v})\subset (S_{e_2}\circ S_{e_3}^{k}\circ S_{e_4})(I_u)$.

As shown in Figure \ref{induction1}(b)(ii), $S(I_{bg_v})\subset (S_{e_2}\circ S_{e_3}^{k}\circ S_{e_4})(I_u)$ means that
\begin{equation*}
S(I_{bg_v}) \cup (S \circ S_{e_2} \circ  S_{e_4})(I_u)  \subset  (S_{e_2}\circ S_{e_3}^{k}\circ S_{e_4})(I_u)
\end{equation*}
and so $rb(g_v + d)\leq bc^kd$ with $r<c^k$. As can be seen from Figure \ref{induction1}, it is also the case that $I_{bc^kg_v}\subset S(I_u)$ which implies
\begin{align*}
bc^kg_v &\leq \max G(S(F_u)) \\
&= \max \left\{rg_u, rbg_v\right\} &&(\text{by Lemma \ref{P2I}(c)}) \\
&= rbg_v  &&(\text{by } \eqref{P2OA}) \\
&< bc^kg_v 
\end{align*}
which is another contradiction.

(b)(iii) $S(I_{bg_v})\subset  (S_{e_2}\circ S_{e_3}^{k})(I_v)$ \emph{and} $S(I_{bg_v})\subset (S_{e_2}\circ S_{e_3}^{k+1})(I_v)$.

The fact that $S(I_{bg_v})\subset (S_{e_2}\circ S_{e_3}^{k+1})(I_v)$ means that, as shown in Figure \ref{induction1}(b)(iii), $I_{bc^kg_v} \subset (S\circ S_{e_2}\circ S_{e_4})(I_u) \subset (S_{e_2}\circ S_{e_3}^{k})(I_v)$. We now obtain a contradiction by applying the proof given in (a)(i).

We have considered all the possible cases for (a) and (b) and shown that none of them can actually occur. This leaves case (c) and $S(I_{bg_v})\subset  S_{e_1}(I_u)$ which implies
\begin{equation*}
I_{g_u}\subset (S\circ (S_{e_2}\circ S_{e_4})^1)(I_u).
\end{equation*}
This proves $P(1)$.

\medskip

\begin{figure}[htb]
\begin{center}
\includegraphics[trim = 12mm 180mm 18mm 15mm, clip, scale =0.7]{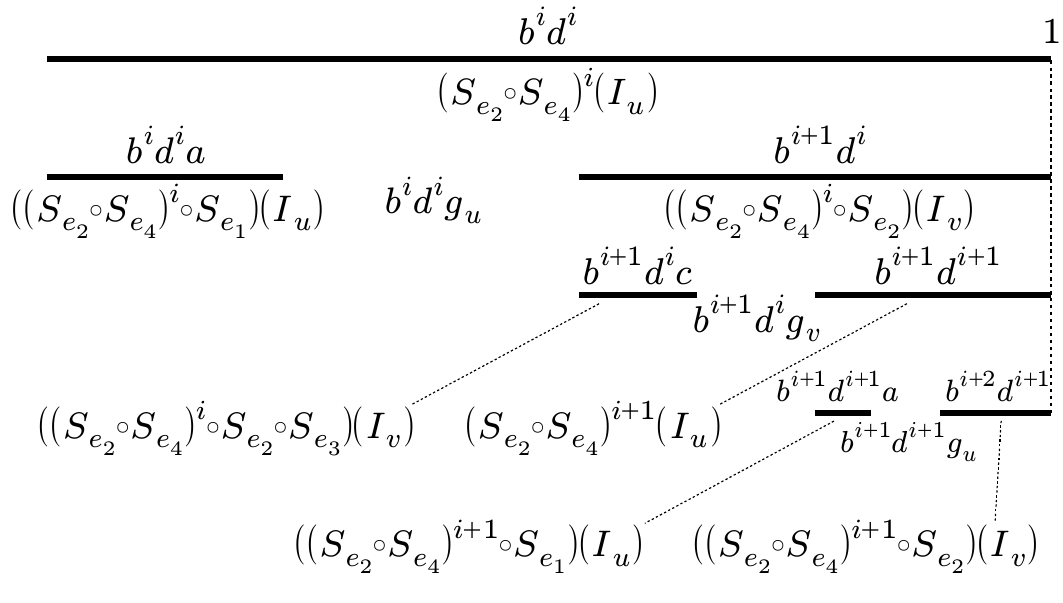}
\end{center}
\caption{The gap intervals $I_{b^{i+1}d^ig_v}$ and $I_{b^{i+1}d^{i+1}g_u}$ together with other subintervals of  $(S_{e_2}\circ S_{e_4})^i(I_u)$. All the intervals shown are contained in $I_u$.}
\label{induction2}
\end{figure}

\emph{Induction hypothesis}.

For $i \in \mathbb{N}$ we assume $P(i)$ is true so that
\begin{equation*}
I_{g_u}\subset (S\circ (S_{e_2}\circ S_{e_4})^i)(I_u).
\end{equation*}

\medskip

\emph{Induction step}.

We follow the proof given for $P(1)$ but use $I_{b^{i+1}d^ig_v}$ in place of $I_{bg_v}$. The gap interval $I_{b^{i+1}d^ig_v}$ is shown in Figure \ref{induction2}. The reader may still consult Figure \ref{induction1} in each of the following cases by simply replacing the lengths $r(a+g_u+bc)$, $rbg_v$ and $rbd$ shown there by $r(1-b^{i+1}d^ig_v-b^{i+1}d^{i+1})$, $rb^{i+1}d^ig_v$ and $rb^{i+1}d^{i+1}$ respectively. These interval lengths, before the map by $S$, are also illustrated in Figure \ref{induction2}. In what follows we replace $S(I_{bg_v})$ and $(S\circ S_{e_2}\circ S_{e_4})(I_u)$ by their counterparts $S(I_{b^{i+1}d^ig_v})$ and $(S\circ (S_{e_2}\circ S_{e_4})^{i+1})(I_u)$.

As before there are just three possible destinations for $I_{b^{i+1}d^ig_v}$ under $S$
\begin{align*}
&\textup{(a)} \quad I_{g_u}\subset S(I_{b^{i+1}d^ig_v}), \\
&\textup{(b)} \quad S(I_{b^{i+1}d^ig_v})\subset  (S_{e_2}\circ S_{e_3}^{k})(I_v), \\
&\textup{(c)} \quad S(I_{b^{i+1}d^ig_v})\subset  S_{e_1}(I_u).
\end{align*}
Again we show that neither (a) nor (b) can happen.

(a)(i) $ I_{g_u}\subset S(I_{b^{i+1}d^ig_v})$ \emph{and} $I_{bc^kg_v} \subset (S\circ (S_{e_2}\circ S_{e_4})^{i+1})(I_u) \subset (S_{e_2}\circ S_{e_3}^{k})(I_v)$.

Here $rb^{i+1}d^{i+1} \leq bc^k$ and
\begin{align*}
bc^kg_v &\leq \max G((S\circ (S_{e_2}\circ S_{e_4})^{i+1})(F_u)) \\
&= \max \left\{rb^{i+1}d^{i+1}g_u, rb^{i+2}d^{i+1}g_v\right\} &&(\text{by Lemma \ref{P2I}(c)}) \\
&=rb^{i+2}d^{i+1}g_v  &&(\text{by } \eqref{P2OA}) \\
&\leq b^2c^kg_v
\end{align*}
which is a contradiction.

(a)(ii) $ I_{g_u}\subset S(I_{b^{i+1}d^ig_v})$ \emph{and} $(S\circ (S_{e_2}\circ S_{e_4})^{i+1})(I_u) \subset (S_{e_2}\circ S_{e_3}^{k}\circ S_{e_4})(I_u)$.

It follows that $rb^{i+1}d^{i+1} \leq bc^kd$ so that $rb^{i}d^{i} \leq c^k$. Also
\begin{equation*}
I_{g_u}\cup  (S_{e_2}\circ S_{e_3}^{k+1})(I_v) \cup I_{bc^kg_v} \subset S(I_{b^{i+1}d^ig_v})
\end{equation*}
so that
\begin{equation*}
g_u + bc^{k+1} + bc^kg_v \leq rb^{i+1}d^ig_v \leq bc^kg_v
\end{equation*}
which is impossible.

(b)(i) $S(I_{b^{i+1}d^ig_v})\subset  (S_{e_2}\circ S_{e_3}^{k})(I_v)$, $(S\circ (S_{e_2}\circ S_{e_4})^{i+1})(I_u) \subset  (S_{e_2}\circ S_{e_3}^{k}\circ S_{e_4})(I_u)$  \emph{and} $I_{bc^kg_v} \subset S(I_{b^{i+1}d^ig_v})$.

Here $bc^kg_v \leq rb^{i+1}d^ig_v$ and $rb^{i+1}d^{i+1} \leq bc^kd$ implies $rb^{i}d^{i}=c^k$ which forces the following equalities
\begin{align*}
(S\circ (S_{e_2}\circ S_{e_4})^{i+1})(I_u) &= (S_{e_2}\circ S_{e_3}^{k}\circ S_{e_4})(I_u) \\
S(I_{b^{i+1}d^ig_v}) &= I_{bc^kg_v} \\
 (S\circ (S_{e_2}\circ S_{e_4})^{i} \circ S_{e_2}\circ S_{e_3})(I_v) &=  (S_{e_2}\circ S_{e_3}^{k+1})(I_v)
\end{align*}
which imply $S\circ (S_{e_2}\circ S_{e_4})^{i} \circ S_{e_2} = S_{e_2}\circ S_{e_3}^{k}$. However immediately to the left of  $((S_{e_2}\circ S_{e_4})^{i} \circ S_{e_2})(I_v)$ is the gap interval $I_{b^{i}d^{i}g_u}$, see Figure \ref{induction2}, and  immediately to the left of $(S_{e_2}\circ S_{e_3}^{k})(I_v)$ is the gap interval $I_{g_u}$, see Figure \ref{induction1}. Therefore it must be the case that
\begin{equation*}
I_{g_u} \subset S(I_{b^{i}d^{i}g_u})
\end{equation*}
which can't happen.

(b)(ii) $S(I_{b^{i+1}d^ig_v})\subset  (S_{e_2}\circ S_{e_3}^{k})(I_v)$ \emph{and} $S(I_{b^{i+1}d^ig_v})\subset (S_{e_2}\circ S_{e_3}^{k}\circ S_{e_4})(I_u)$.

In this case $S(I_{b^{i+1}d^ig_v})\subset (S_{e_2}\circ S_{e_3}^{k}\circ S_{e_4})(I_u)$ implies that
\begin{equation*}
S(I_{b^{i+1}d^ig_v}) \cup (S\circ (S_{e_2}\circ S_{e_4})^{i+1})(I_u)  \subset  (S_{e_2}\circ S_{e_3}^{k}\circ S_{e_4})(I_u)
\end{equation*}
so $rb^{i+1}d^i(g_v + d)\leq bc^kd$ with $rb^id^i<c^k$. Here we need to explicitly invoke the induction hypothesis which states that $I_{g_u}\subset (S\circ (S_{e_2}\circ S_{e_4})^i)(I_u)$. This ensures that $(S\circ (S_{e_2}\circ S_{e_4})^i)(0) \in S_{e_1}(I_u)$ and $(S\circ (S_{e_2}\circ S_{e_4})^i)(1) = S(1)\in  (S_{e_2}\circ S_{e_3}^{k}\circ S_{e_4})(I_u)$ by \eqref{S1}, so that $I_{bc^kg_v} \subset (S\circ (S_{e_2}\circ S_{e_4})^i)(I_u)$. It follows that
\begin{align*}
bc^kg_v &\leq \max G((S\circ (S_{e_2}\circ S_{e_4})^i)(F_u)) \\
&= \max \left\{rb^id^ig_u, rb^{i+1}d^ig_v\right\} &&(\text{by Lemma \ref{P2I}(c)}) \\
&= rb^{i+1}d^ig_v  &&(\text{by } \eqref{P2OA}) \\
&< bc^kg_v 
\end{align*}
and this is another contradiction.

(b)(iii) $S(I_{b^{i+1}d^ig_v})\subset  (S_{e_2}\circ S_{e_3}^{k})(I_v)$ \emph{and} $S(I_{b^{i+1}d^ig_v})\subset (S_{e_2}\circ S_{e_3}^{k+1})(I_v)$.

The fact that $S(I_{b^{i+1}d^ig_v})\subset (S_{e_2}\circ S_{e_3}^{k+1})(I_v)$ means that $I_{bc^kg_v} \subset (S\circ (S_{e_2}\circ S_{e_4})^{i+1})(I_u)\subset (S_{e_2}\circ S_{e_3}^{k})(I_v)$ and this can't happen as was shown in (a)(i).

This covers all the possible cases and we conclude that neither (a) nor (b) occurs which leaves case (c) and so $ S(I_{b^{i+1}d^ig_v})\subset  S_{e_1}(I_u)$. This is enough to prove that 
\begin{equation*}
I_{g_u}\subset (S\circ (S_{e_2}\circ S_{e_4})^{i+1})(I_u)
\end{equation*}
and so $P(i)$ implies $P(i+1)$ for all $i \in \mathbb{N}$ which completes the induction step.
\medskip

 It follows that $P(i)$ holds for all $i \in \mathbb{N}$. Therefore $g_u=0$ and this is our final contradiction.
\end{proof}

The proof of Lemma \ref{P2P} follows the same pattern as the proof of Lemma \ref{P2O}.
\begin{lem}
\label{P2P}
For the $2$-vertex IFS of Figure \ref{P22VertexUnitInterval}, let $S: \mathbb{R} \rightarrow \mathbb{R}$ be a (non-reflecting) contracting similarity with contracting similarity ratio $r$, $0<r<1$, such that $S(F_u) \subset F_v$. 

Then either $S(I_u)\subset S_{e_3}(I_v)$ or $S(I_u)\subset S_{e_4}(I_u)$.
\end{lem}

\begin{proof} For a contradiction we assume that $S(I_u)$ spans the gap between the level-$1$ intervals with $I_{g_v}\subset S(I_u)$, $S(0) \in S_{e_3}(I_v)$ and $S(1) \in S_{e_4}(I_u)$, as shown in Figure \ref{induction3}.

\begin{figure}[htb]
\begin{center}
\includegraphics[trim = 5mm 125mm 17mm 15mm, clip, scale =0.7]{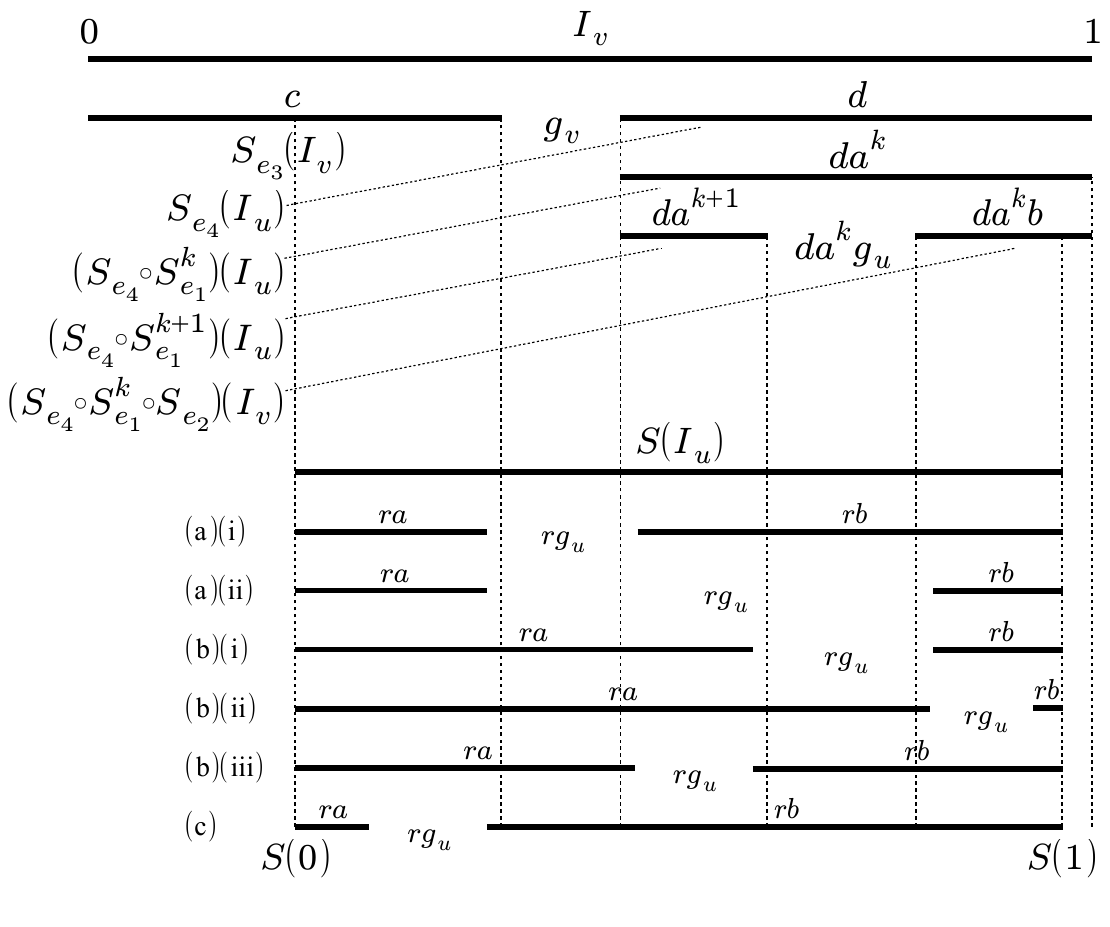}
\end{center}
\caption{The possible cases for Lemma \ref{P2P}. In the diagram $k=0$.}
\label{induction3}
\end{figure}

As $I_{g_v}\subset S(I_u)$ it follows that
\begin{align*}
g_v &\leq \max G(S(F_u)) \\
&= \max \left\{rg_u, rbg_v\right\} &&(\text{by Lemma \ref{P2I}(c)}), \\
&= rg_u  &&(\text{as }rbg_v < g_v),
\end{align*}
which implies
\begin{equation}
\label{P2PA}
bg_v < g_v < g_u
\end{equation}
and so $I_{g_u}$ is the unique maximum length gap interval of $F_u$. 

As there are points of  $F_u$ as close as we like to $1$ on its left and there are no such points immediately to the left of  $c+g_v$ it follows that $S(1) > c+g_v$ and there exists a $k \in \mathbb{N}\cup \left\{0\right\}$ such that 
\begin{equation*}
c+g_v+da^{k+1} = (S_{e_4}\circ S_{e_1}^{k+1})(1) < S(1) \leq (S_{e_4}\circ S_{e_1}^{k})(1) =  c+g_v+da^{k}.
\end{equation*}
For such $k\in \mathbb{N}\cup \left\{0\right\}$ it must be the case that 
\begin{equation}
\label{S12}
S(1) \in (S_{e_4}\circ S_{e_1}^k\circ S_{e_2})(I_v).
\end{equation}
The situation is the same as that shown in  Figure \ref{P2b} if we replace the symbols $b, c, d, S_{e_2}, S_{e_3}, S_{e_4}, I_v, I_u$ by $d, a, b, S_{e_4}, S_{e_1}, S_{e_2}, I_u, I_v$ respectively. For the rest of this proof we consider $k$  in \eqref{S12} to be fixed. An example with $k=0$ is illustrated in Figure \ref{induction3}.

For $i\in \mathbb{N}\cup \left\{0\right\}$ let $Q(i)$ be the statement
\begin{equation*}
I_{g_v}\subset (S\circ (S_{e_2} \circ S_{e_4})^i \circ S_{e_2})(I_v).
\end{equation*}
Again we prove by induction that $Q(i)$ holds for all $i\in  \mathbb{N}\cup \left\{0\right\}$ which is enough for a contradiction as it forces $g_v=0$.
\medskip

\emph{Induction base}.

There are just three possible destinations for $I_{g_u}$ under $S$
\begin{align*}
&\textup{(a)} \quad I_{g_v}\subset S(I_{g_u}), \\
&\textup{(b)} \quad S(I_{g_u})\subset  (S_{e_4}\circ S_{e_1}^{k})(I_u), \\
&\textup{(c)} \quad S(I_{g_u})\subset  S_{e_3}(I_v).
\end{align*}
We prove that neither (a) nor (b) can happen which leaves (c). This is enough to prove $Q(0)$ as should be clear from Figure \ref{induction3}(c). We consider all the possibilities that can arise for (a) and (b) and as we arrive at a contradiction in each of these cases we won't point this out again. 

(a)(i) $ I_{g_v}\subset S(I_{g_u})$ \emph{and} $I_{da^kg_u} \subset (S\circ S_{e_2})(I_v) \subset (S_{e_4}\circ S_{e_1}^{k})(I_u)$.

This is the situation shown in Figure \ref{induction3}(a)(i). Here $rb \leq da^k$ and it must be the case that
\begin{align*}
da^kg_u &\leq \max G((S\circ S_{e_2})(F_v)) \\
&= \max \left\{rbg_v, rbdg_u\right\} &&(\text{by Lemma \ref{P2I}(d)}) \\
&= rbg_v  &&(\text{as } rbdg_u \leq d^2a^kg_u < da^kg_u)  \\
&\leq da^kg_v   \\
&< da^kg_u  &&(\text{by } \eqref{P2PA}) .
\end{align*}

(a)(ii) $ I_{g_v}\subset S(I_{g_u})$ \emph{and} $(S\circ S_{e_2})(I_v) \subset (S_{e_4}\circ S_{e_1}^{k}\circ S_{e_2})(I_v)$.

Here $rb \leq da^kb$ so that $r \leq da^k$. As illustrated in Figure \ref{induction3}(a)(ii)
\begin{equation*}
I_{g_v}\cup  (S_{e_4}\circ S_{e_1}^{k+1})(I_u) \cup I_{da^kg_u} \subset S(I_{g_u})
\end{equation*}
and so 
\begin{equation*}
g_v + da^{k+1} + da^kg_u \leq rg_u \leq da^kg_u.
\end{equation*}

(b)(i) $S(I_{g_u})\subset  (S_{e_4}\circ S_{e_1}^{k})(I_u)$, $(S \circ S_{e_2})(I_v) \subset  (S_{e_4}\circ S_{e_1}^{k}\circ S_{e_2})(I_v)$  \emph{and} $I_{da^kg_u} \subset S(I_{g_u})$.

As illustrated in  Figure \ref{induction3}(b)(i), in this case $da^kg_u \leq rg_u$ and $rb \leq da^kb$ so that $r=da^k$. This forces the following equalities
\begin{align*}
(S \circ S_{e_2})(I_v) &= (S_{e_4}\circ S_{e_1}^{k}\circ S_{e_2})(I_v) \\
S(I_{g_u}) &= I_{da^kg_u} \\
 (S \circ S_{e_1})(I_u) &=  (S_{e_4}\circ S_{e_1}^{k+1})(I_u)
\end{align*}
which means that $S = S_{e_4}\circ S_{e_1}^{k}$. It follows that $S(0)=(S_{e_4}\circ S_{e_1}^{k})(0)=S_{e_4}(0)$ and so $S(0) \notin S_{e_3}(I_v)$.

(b)(ii) $S(I_{g_u})\subset  (S_{e_4}\circ S_{e_1}^{k})(I_u)$ \emph{and} $S(I_{g_u})\subset (S_{e_4}\circ S_{e_1}^{k}\circ S_{e_2})(I_v)$.

As shown in Figure \ref{induction3}(b)(ii), $S(I_{g_u})\subset (S_{e_4}\circ S_{e_1}^{k}\circ S_{e_2})(I_v)$ implies
\begin{equation*}
S(I_{g_u}) \cup (S \circ S_{e_2})(I_v)  \subset (S_{e_4}\circ S_{e_1}^{k}\circ S_{e_2})(I_v)
\end{equation*}
and so $r(g_u + b)\leq da^kb$ with $r<da^k$. It is also the case that $I_{da^kg_u}\subset S(I_u)$, see Figure \ref{induction3}, and we must have
\begin{align*}
da^kg_u &\leq \max G(S(F_u)) \\
&= \max \left\{rg_u, rbg_v\right\} &&(\text{by Lemma \ref{P2I}(c)}) \\
&= rg_u  &&(\text{by } \eqref{P2PA}) \\
&< da^kg_u.
\end{align*}

(b)(iii) $S(I_{g_u})\subset  (S_{e_4}\circ S_{e_1}^{k})(I_u)$ \emph{and} $S(I_{g_u})\subset (S_{e_4}\circ S_{e_1}^{k+1})(I_u)$.

In this situation $I_{da^kg_u} \subset (S\circ S_{e_2})(I_v) \subset (S_{e_4}\circ S_{e_1}^{k})(I_u)$,  see Figure \ref{induction3}(b)(iii), and we can apply the proof of (a)(i).

We have considered all the possible cases for (a) and (b) and shown that none of them can happen. This leaves case (c) and so $S(I_{g_u})\subset  S_{e_3}(I_v)$ which implies
\begin{equation*}
I_{g_v}\subset (S\circ (S_{e_2} \circ S_{e_4})^0 \circ S_{e_2})(I_v)
\end{equation*}
and proves $Q(0)$.

\medskip

\emph{Induction hypothesis}.

For $i \in \mathbb{N}\cup\left\{0\right\}$ we assume $Q(i)$ is true so that
\begin{equation*}
I_{g_v}\subset (S\circ (S_{e_2} \circ S_{e_4})^i \circ S_{e_2})(I_v).
\end{equation*}

\medskip

\emph{Induction step}.

The proof mirrors that for $Q(0)$ but we use $I_{b^{i+1}d^{i+1}g_u}$ in place of $I_{g_u}$. The gap interval $I_{b^{i+1}d^{i+1}g_u}$ is shown in Figure \ref{induction2}. The reader may still refer to Figure \ref{induction3} in each of the following cases if the lengths $ra$, $rg_u$ and $rb$ shown there are replaced by $r(1-b^{i+1}d^{i+1}g_u-b^{i+2}d^{i+1})$, $rb^{i+1}d^{i+1}g_u$ and $rb^{i+2}d^{i+1}$ respectively. These interval lengths, before the map by $S$, are illustrated in Figure \ref{induction2} and are contained in $I_u$. In what follows we replace $S(I_{g_u})$ and $(S\circ S_{e_2})(I_v)$ in the induction base by $S(I_{b^{i+1}d^{i+1}g_u})$ and $(S\circ (S_{e_2}\circ S_{e_4})^{i+1} \circ S_{e_2})(I_v)$ respectively.

As before there are just three possible destinations for $I_{b^{i+1}d^{i+1}g_u}$ under $S$
\begin{align*}
&\textup{(a)} \quad I_{g_v}\subset S(I_{b^{i+1}d^{i+1}g_u}), \\
&\textup{(b)} \quad S(I_{b^{i+1}d^{i+1}g_u})\subset  (S_{e_4}\circ S_{e_1}^{k})(I_u), \\
&\textup{(c)} \quad S(I_{b^{i+1}d^{i+1}g_u})\subset  S_{e_3}(I_v).
\end{align*}
As in the induction base, in each of the following cases we arrive at a contradiction.

(a)(i) $I_{g_v}\subset S(I_{b^{i+1}d^{i+1}g_u})$ \emph{and} $I_{da^kg_u} \subset (S\circ (S_{e_2}\circ S_{e_4})^{i+1} \circ S_{e_2})(I_v) \subset (S_{e_4}\circ S_{e_1}^{k})(I_u)$.

Here $rb^{i+2}d^{i+1} \leq da^k$ and it must be the case that
\begin{align*}
da^kg_u &\leq \max G((S\circ (S_{e_2}\circ S_{e_4})^{i+1} \circ S_{e_2})(F_v)) \\
&= \max \left\{rb^{i+2}d^{i+1}g_v, rb^{i+2}d^{i+2}g_u\right\} &&(\text{by Lemma \ref{P2I}(d)}) \\
&= rb^{i+2}d^{i+1}g_v  &&(\text{as }  rb^{i+2}d^{i+2}g_u \leq d^{2}a^kg_u < da^kg_u)  \\
&\leq da^kg_v   \\
&< da^kg_u  &&(\text{by } \eqref{P2PA}) .
\end{align*}

(a)(ii) $I_{g_v}\subset S(I_{b^{i+1}d^{i+1}g_u})$ \emph{and} $(S\circ (S_{e_2}\circ S_{e_4})^{i+1} \circ S_{e_2})(I_v)  \subset (S_{e_4}\circ S_{e_1}^{k}\circ S_{e_2})(I_v)$.

In this case $rb^{i+2}d^{i+1} \leq da^kb$ so that $rb^{i+1}d^{i+1} \leq da^k$. Also
\begin{equation*}
I_{g_v}\cup  (S_{e_4}\circ S_{e_1}^{k+1})(I_u) \cup I_{da^kg_u} \subset S(I_{b^{i+1}d^{i+1}g_u})
\end{equation*}
and so 
\begin{equation*}
g_v + da^{k+1} + da^kg_u \leq rb^{i+1}d^{i+1}g_u \leq da^kg_u.
\end{equation*}

(b)(i) $S(I_{b^{i+1}d^{i+1}g_u})\subset  (S_{e_4}\circ S_{e_1}^{k})(I_u)$, $(S\circ (S_{e_2}\circ S_{e_4})^{i+1} \circ S_{e_2})(I_v) \subset  (S_{e_4}\circ S_{e_1}^{k}\circ S_{e_2})(I_v)$  \emph{and} $I_{da^kg_u} \subset S(I_{b^{i+1}d^{i+1}g_u})$.

Here $da^kg_u \leq rb^{i+1}d^{i+1}g_u$ and $rb^{i+2}d^{i+1} \leq da^kb$ which implies $rb^{i+1}d^{i+1}=da^k$. This forces the following equalities
\begin{align*}
(S\circ (S_{e_2}\circ S_{e_4})^{i+1} \circ S_{e_2})(I_v)  &= (S_{e_4}\circ S_{e_1}^{k}\circ S_{e_2})(I_v) \\
S(I_{b^{i+1}d^{i+1}g_u}) &= I_{da^kg_u} \\
(S\circ (S_{e_2}\circ S_{e_4})^{i+1} \circ S_{e_1})(I_u)  &=  (S_{e_4}\circ S_{e_1}^{k+1})(I_u)
\end{align*}
which means that $S\circ (S_{e_2}\circ S_{e_4})^{i+1}= S_{e_4}\circ S_{e_1}^{k}$. As shown in Figure \ref{induction2} the gap interval immediately to the left of $(S_{e_2}\circ S_{e_4})^{i+1}(I_u)$ is $I_{b^{i+1}d^ig_v}$ and as shown in Figure \ref{induction3} the gap interval immediately to the left of $(S_{e_4}\circ S_{e_1}^{k})(I_u)$ is $I_{g_v}$. Therefore
\begin{equation*}
I_{g_v} \subset S(I_{b^{i+1}d^ig_v}).
\end{equation*}

(b)(ii) $S(I_{b^{i+1}d^{i+1}g_u})\subset  (S_{e_4}\circ S_{e_1}^{k})(I_u)$ \emph{and} $S(I_{b^{i+1}d^{i+1}g_u})\subset (S_{e_4}\circ S_{e_1}^{k}\circ S_{e_2})(I_v)$.

Here $S(I_{b^{i+1}d^{i+1}g_u})\subset (S_{e_4}\circ S_{e_1}^{k}\circ S_{e_2})(I_v)$ implies
\begin{equation*}
S(I_{b^{i+1}d^{i+1}g_u}) \cup (S\circ (S_{e_2}\circ S_{e_4})^{i+1} \circ S_{e_2})(I_v)  \subset (S_{e_4}\circ S_{e_1}^{k}\circ S_{e_2})(I_v)
\end{equation*}
and so $rb^{i+1}d^{i+1}(g_u + b)\leq da^kb$ with $rb^{i+1}d^{i+1}<da^k$. By the induction hypothesis $ (S\circ (S_{e_2}\circ S_{e_4})^{i} \circ S_{e_2})(0) \in S_{e_3}(I_v)$ and $ (S\circ (S_{e_2}\circ S_{e_4})^{i} \circ S_{e_2})(1) = S(1) \in (S_{e_4}\circ S_{e_1}^{k}\circ S_{e_2})(I_v)$ by \eqref{S12}. This ensures $I_{da^kg_u} \subset  (S\circ (S_{e_2}\circ S_{e_4})^{i} \circ S_{e_2})(I_v)$ and so it must be the case that
\begin{align*}
da^kg_u &\leq \max G( (S\circ (S_{e_2}\circ S_{e_4})^{i} \circ S_{e_2})(F_v)) \\
&= \max \left\{rb^{i+1}d^{i}g_v, rb^{i+1}d^{i+1}g_u\right\} &&(\text{by Lemma \ref{P2I}(d)}) \\
&= rb^{i+1}d^{i}g_v  &&(\text{as } rb^{i+1}d^{i+1}g_u < da^kg_u) \\
&< g_v.
\end{align*}
Also by the induction hypothesis $I_{g_v}\subset (S\circ (S_{e_2} \circ S_{e_4})^i \circ S_{e_2})(I_v)$ which implies
\begin{align*}
g_v &\leq \max G( (S\circ (S_{e_2}\circ S_{e_4})^{i} \circ S_{e_2})(F_v)) \\
&= \max \left\{rb^{i+1}d^{i}g_v, rb^{i+1}d^{i+1}g_u\right\} &&(\text{by Lemma \ref{P2I}(d)}) \\
&= rb^{i+1}d^{i+1}g_u  &&(\text{as } rb^{i+1}d^{i}g_v < g_v) \\
&< da^kg_u &&(\text{as } rb^{i+1}d^{i+1}<da^k).
\end{align*}

(b)(iii)$S(I_{b^{i+1}d^{i+1}g_u})\subset  (S_{e_4}\circ S_{e_1}^{k})(I_u)$ \emph{and} $S(I_{b^{i+1}d^{i+1}g_u})\subset (S_{e_4}\circ S_{e_1}^{k+1})(I_u)$.

Here $S(I_{b^{i+1}d^{i+1}g_u})\subset (S_{e_4}\circ S_{e_1}^{k+1})(I_u)$ implies $I_{da^kg_u} \subset  (S\circ (S_{e_2}\circ S_{e_4})^{i+1} \circ S_{e_2})(I_v)  \subset (S_{e_4}\circ S_{e_1}^{k})(I_u)$ and so the proof of (a)(i) applies.

We have now shown that none of the possible cases for (a) and (b) can actually happen which leaves case (c) and so $S(I_{b^{i+1}d^{i+1}g_u})\subset  S_{e_3}(I_v)$. It must be the case, see Figure \ref{induction2}, that
\begin{equation*}
I_{g_v}\subset (S\circ (S_{e_2} \circ S_{e_4})^{i+1} \circ S_{e_2})(I_v).
\end{equation*}
This completes the induction step since we have shown $Q(i)$ implies $Q(i+1)$ for all $i\in \mathbb{N}\cup \left\{0\right\}$.

\medskip

Therefore $Q(i)$ is true for all $i\in \mathbb{N}\cup \left\{0\right\}$ and $g_v=0$ which is our final contradiction.
\end{proof}
The symmetry of the $2$-vertex IFS of Figure \ref{P22VertexUnitInterval} means we have also proved the next two lemmas.
\begin{lem}
\label{P2Ob}
For the $2$-vertex IFS of Figure \ref{P22VertexUnitInterval}, let $S: \mathbb{R} \rightarrow \mathbb{R}$ be a (non-reflecting) contracting similarity with contracting similarity ratio $r$, $0<r<1$, such that $S(F_v) \subset F_v$. 

Then either $S(I_v)\subset S_{e_3}(I_v)$ or $S(I_v)\subset S_{e_4}(I_u)$.
\end{lem}

\begin{lem}
\label{P2Pb}
For the $2$-vertex IFS of Figure \ref{P22VertexUnitInterval}, let $S: \mathbb{R} \rightarrow \mathbb{R}$ be a (non-reflecting) contracting similarity with contracting similarity ratio $r$, $0<r<1$, such that $S(F_v) \subset F_u$. 

Then either $S(I_v)\subset S_{e_1}(I_u)$ or $S(I_v)\subset S_{e_2}(I_v)$.
\end{lem}
The hard work in the proofs of the preceding lemmas pays off in the proof of Lemma \ref{P2subset1} which we use in the proof of Theorem \ref{P2M} in Subsection \ref{three1}. For an example with $F_u\subset F_v$ and $F_u\neq F_v$ where $F_u$ is a standard IFS attractor see Subsection \ref{five3}.
\begin{lem}
\label{P2subset1}
For the attractor $(F_u,F_v)$ of the $2$-vertex IFS of Figure \ref{P22VertexUnitInterval}, if $F_u\subset F_v$ then $F_u=F_v$.
\end{lem}
\begin{proof}
From Equation \eqref{Invariance equation1}, $S_{e_2}(F_v)\subset F_v$ and by Lemma \ref{P2Ob}, $S_{e_2}(F_v)\subset S_{e_2}(I_v)\subset S_{e_4}(I_u)$, which means $S_{e_2}(F_v)\subset F_v \cap S_{e_4}(I_u) = S_{e_4}(F_u)$ by Equation \eqref{Invariance equation2} and the CSSC. Consider the position of $S_{e_2}(0)$. Clearly $S_{e_4}(0) \leq S_{e_2}(0) < (S_{e_4}\circ S_{e_2})(0)$. If $S_{e_4}(0) < S_{e_2}(0) < (S_{e_4}\circ S_{e_2})(0)$ then $S_{e_4}^{-1}\circ S_{e_2}$ is a contracting similarity such that $(S_{e_4}^{-1}\circ S_{e_2})(F_v)\subset F_u$ where $(S_{e_4}^{-1}\circ S_{e_2})(I_v)$ spans the gap between the level-$1$ intervals of $F_u$, but this is impossible by Lemma \ref{P2Pb}. This implies $S_{e_2}(0) = S_{e_4}(0)$, so that $S_{e_2}=S_{e_4}$ and $F_v\subset F_u$. Therefore $F_u=F_v$.
\end{proof}
Again a symmetrical argument means that we also have the next lemma.
\begin{lem}
\label{P2subset2}
For the attractor $(F_u,F_v)$ of the $2$-vertex IFS of Figure \ref{P22VertexUnitInterval} if $F_v\subset F_u$ then $F_v=F_u$.
\end{lem}

Although we will not make use of Lemma \ref{P2levelk} it's worth stating as it is a consequence of Lemmas \ref{P2O}-\ref{P2subset2}. It can also be used to provide another proof of Theorem \ref{P2M}.
\begin{lem}
\label{P2levelk}
For the $2$-vertex IFS of Figure \ref{P22VertexUnitInterval}, let $S: \mathbb{R} \rightarrow \mathbb{R}$ be a (non-reflecting) contracting similarity with contracting similarity ratio $r$, $0<r<1$, which maps a component of  the attractor $(F_u,F_v)$ into a component of $(F_u,F_v)$. Suppose also that $F_u\neq F_v$. 

Then $S(F_u)$ \textup{(}or $S(F_v)$\textup{)} is an elementary piece. More precisely
\begin{align*}
&\textup{(a)}  \quad \textrm{if } S(F_u)\subset F_u \textrm{ then } S(F_u) = S_{\be}(F_{t(\be)}) = S_{\be}(F_u) \textrm{ for some } {\be} \in E_{uu}^* , \\
&\textup{(b)}  \quad \textrm{if } S(F_u)\subset F_v \textrm{ then } S(F_u) = S_{\be}(F_{t(\be)}) = S_{\be}(F_u) \textrm{ for some } {\be} \in E_{vu}^* , \\
&\textup{(c)}  \quad \textrm{if } S(F_v)\subset F_v \textrm{ then } S(F_v) = S_{\be}(F_{t(\be)}) = S_{\be}(F_v) \textrm{ for some } {\be} \in E_{vv}^* , \\
&\textup{(d)}  \quad \textrm{if } S(F_v)\subset F_u \textrm{ then } S(F_v) = S_{\be}(F_{t(\be)}) = S_{\be}(F_v) \textrm{ for some } {\be} \in E_{uv}^* . 
\end{align*} 
\end{lem}
\begin{proof}
(a) For a contradiction suppose $S(I_u) \neq S_{\be}(I_{t(\be)})$ for any ${\be} \in E_u^*$. If $S(I_u)$ is strictly contained in a level-$k$ interval for each $k \in \mathbb{N}$ then $r=0$, so there exists $k \in \mathbb{N}$ and $\Bf \in E_u^k$ such that $S(I_u) \subsetneqq S_{\Bf}(I_{t(\Bf)})$ with $S(I_u)$ spanning the gap between the two level-$(k+1)$ subintervals of $S_{\Bf}(I_{t(\Bf)})$. It follows that $(S_{\Bf}^{-1}\circ S)(I_u)$ is a (non-reflecting) contracting similarity which spans the gap between level-$1$ intervals. This is impossible by Lemmas \ref{P2O} and \ref{P2P} and proves $S(I_u) = S_{\be}(I_{t(\be)})$ for some ${\be} \in E_u^*$. 

If ${\be} \in E_{uv}^*$ then $S(I_u) = S_{\be}(I_v)$, $F_u \subset F_v$ and $F_u = F_v$ by Lemma \ref{P2subset1}. Therefore ${\be} \in E_{uu}^*$. 

The proofs of (b), (c) and (d) are similar.
\end{proof}
If we add any number of edges (with non-reflecting similarities) to the directed graph of Figure \ref{P22VertexUnitInterval}, maintaining the CSSC, keeping all level-$1$ gap lengths equal across both vertices and ensuring that $F_u \not\subset F_v$ and $F_v \not\subset F_u$, then the conclusions of Lemma \ref{P2levelk} will still hold. This result can naturally be extended to allow for reflecting similarities and to $n$-vertex (CSSC) IFSs.

\subsection{Proof of Theorem \ref{P2M}} \label{three1}

\begin{proof}
For a contradiction we assume $F_u$ is the attractor of a standard IFS, that is we assume $F_u$ satisfies an invariance equation of the form
\begin{equation*}
\label{onevertexequation}
F_u=\bigcup_{i=1}^n S_i(F_u)
\end{equation*}
for some $n\geq2$ where each $S_i$ is a (non-reflecting) contracting similarity. By Lemma \ref{P2O}, for each $i$, either $S_i(F_u)\subset S_i(I_u)\subset S_{e_1}(I_u)$ or $S_i(F_u)\subset S_i(I_u)\subset S_{e_2}(I_v)$ so the similarities  split into two groups. After relabelling, we may now write $F_u$ as
\begin{equation*}
F_u=\bigcup_{i=1}^{m_1} T_i(F_u)  \ \cup  \ \bigcup_{i=1}^{n_1} S_i^\prime(F_u)
\end{equation*}
with $m_1+n_1=n$, $1\leq m_1,n_1 < n$, $\bigcup_{i=1}^{m_1} T_i(F_u) \subset S_{e_1}(I_u)$ and $\bigcup_{i=1}^{n_1} S_i^\prime(F_u) \subset S_{e_2}(I_v)$. In fact by Equation \eqref{Invariance equation1} and the CSSC it follows that  $S_{e_1}(F_u)=\bigcup_{i=1}^{m_1} T_i(F_u)$ and $S_{e_2}(F_v)=\bigcup_{i=1}^{n_1} S_i^\prime(F_u)$.

We now concentrate on the second of these equations looking to exploit the fact that $F_u \neq F_v$. If $n_1=1$ then $S_1^\prime(F_u)=S_{e_2}(F_v)$ for some single similarity $S_1^\prime$ which implies $S_1^\prime(0)=S_{e_2}(0)$, $S_1^\prime(1)=S_{e_2}(1)$,  so that $S_1^\prime=S_{e_2}$ and $F_u=F_v$. This contradiction means that $1<n_1< n$. If for any $i$, $\left|S_i^\prime(F_u)\right|=\left|S_{e_2}(F_v)\right|$ then as  $S_i^\prime(F_u) \subset S_{e_2}(F_v)$ it follows that $S_i^\prime(0)=S_{e_2}(0)$, $S_i^\prime(1)=S_{e_2}(1)$, so that $S_i^\prime=S_{e_2}$ and $F_u \subset F_v$. This means $F_u = F_v$ by Lemma \ref{P2subset1} and again contradicts the assumption that $F_u \neq F_v$. Therefore $\left|S_i^\prime(F_u)\right|<\left|S_{e_2}(F_v)\right|$ and the maps $S_{e_2}^{-1}\circ S_i^\prime$ are contracting similarities. Relabelling we can now write $F_v= \bigcup_{i=1}^{n_1} (S_{e_2}^{-1}\circ S_i^\prime)(F_u)$ as
\begin{equation*}
F_v= \bigcup_{i=1}^{n_1} S_{1,i}(F_u).
\end{equation*}
By Lemma \ref{P2P} the contracting similarities $S_{1,i}$ must also split into two groups. Again after relabelling we obtain 
\begin{equation*}
F_v=\bigcup_{i=1}^{n_2}S_{1,i}^\prime(F_u)  \ \cup  \ \bigcup_{i=1}^{m_2} T_{1,i}(F_u).
\end{equation*}
where $n_2+m_2=n_1$, $1\leq m_2,n_2 < n_1$, $\bigcup_{i=1}^{n_2}S_{1,i}^\prime(F_u) \subset S_{e_3}(I_v)$ and $\bigcup_{i=1}^{m_2} T_{1,i}(F_u)$  $\subset S_{e_4}(I_u)$. Equation \eqref{Invariance equation2} and the CSSC imply $S_{e_3}(F_v)=\bigcup_{i=1}^{n_2}S_{1,i}^\prime(F_u)$. Exactly as argued above, using Lemma \ref{P2subset1}, it follows that $1< n_2 < n_1 < n$ and that each $S_{e_3}^{-1}\circ S_{1,i}^\prime$ is a contracting similarity. We now have $F_v=\bigcup_{i=1}^{n_2}(S_{e_3}^{-1}\circ S_{1,i}^\prime)(F_u)$  which we relabel as
\begin{equation*}
F_v= \bigcup_{i=1}^{n_2}S_{2,i}(F_u).
\end{equation*}
So far we have $1< n_2 < n_1 < n$ and it is clear that we can repeat this process indefinitely, but from now on using only Lemmas \ref{P2P} and \ref{P2subset1}, Equation \eqref{Invariance equation2}, and the expanding similarity $S_{e_3}^{-1}$. Each time we obtain an expression for $F_v$ 
\begin{equation*}
F_v= \bigcup_{i=1}^{n_{k+1}}S_{k+1,i}(F_u)
\end{equation*}
where the number of contracting similarities $n_{k+1}$ has been reduced with $1 < n_{k+1}<n_k$. This constructs an infinite sequence $(n_k)$, with $n_k \in \mathbb{N}$ and $1 < n_{k+1} < n_k$ for each $k \in \mathbb{N}$, which is impossible. Therefore $F_u$ is not the attractor of a standard IFS. 

An appeal to symmetry is enough to ensure that $F_v$ cannot be the attractor of any standard IFS either and this completes the proof of Theorem \ref{P2M}. 
\end{proof}

\section{An example}\label{four}

For the class of $2$-vertex IFSs of Figure \ref{P22VertexUnitInterval} the Hausdorff dimension $s=\dimH F_u=\dimH F_v$ is the solution of 
\begin{gather}
\label{HDimension}
(r_{e_1}^t-1)(r_{e_3}^t-1)-r_{e_2}^tr_{e_4}^t = (a^t-1)(c^t-1) - b^td^t = 0,
\end{gather}
see Theorem \ref{dimension}. Clearly neither $0$ nor $1$ is a solution so $0<s<1$. 

We now present just one specific example of the class of $2$-vertex IFSs of Figure \ref{P22VertexUnitInterval} in which the golden ratio makes an appearance. Consider the following parameters 
\begin{equation*}
a=\frac{1}{4}, \ g_u=\frac{1}{4}, \ b=\frac{1}{2}, \ c=\frac{1}{2}, \ g_v=\frac{1}{4}, \ d=\frac{1}{4}.
\end{equation*}
For these parameters the level-$k$ intervals, for $0 \leq k \leq 5$, are illustrated in Figure \ref{Paper2Example}. As given in Equation \eqref{HDimension} the Hausdorff dimension is the solution of
\begin{equation*}
\bigl(\left(\tfrac{1}{4}\right)^t-1\bigr)\bigl(\left(\tfrac{1}{2}\right)^t-1\bigr) - \left(\tfrac{1}{2}\right)^t\left(\tfrac{1}{4}\right)^t = 0.
\end{equation*}

\begin{figure}[htb]
\begin{center}
\includegraphics[trim = 5mm 192mm 5mm 16mm, clip, scale=0.7]{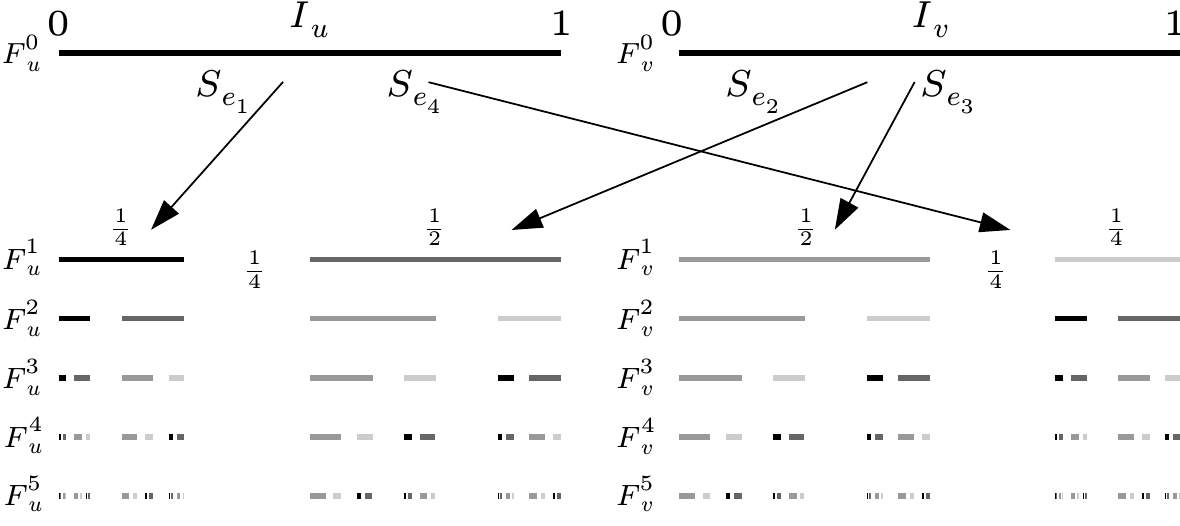}
\end{center}
\caption{Level-$k$ intervals of $F_u$ and $F_v$, for $0 \leq k \leq 5$. Neither $F_u$ nor $F_v$ is the attractor of a standard IFS. }
\label{Paper2Example}
\end{figure} 

\noindent This reduces to a quadratic
\begin{equation}
\label{quadratic}
\bigl(\left(\tfrac{1}{2}\right)^t\bigr)^2 + \left(\tfrac{1}{2}\right)^t - 1 = 0,
\end{equation}
and so the Hausdorff dimension is
\begin{equation*}
s = \frac{\ln \bigl(\frac{-1+\sqrt 5}{2}\bigr)}{\ln \left(\frac{1}{2} \right)} \approx 0.694.
\end{equation*}
Using Equation \eqref{quadratic} we obtain
\begin{gather*}
\frac{1-a^s}{b^s}=\frac{1-\left(\frac{1}{4}\right)^s}{\left(\frac{1}{2}\right)^s} = 1, \\ 
\frac{(1-b)(1-a^s)}{ba^s}=\frac{1-\left(\frac{1}{4}\right)^s}{\left(\frac{1}{4}\right)^s} = \frac{1+\sqrt5}{2} \geq 1.
\end{gather*}

Therefore Conditions (1) and (2) of Theorem \ref{2GthmN} hold and $\mathcal{H}^s(F_u) = \mathcal{H}^s(F_v) =1$. As $F_u \neq F_v$, applying Theorem \ref{P2M}, we conclude that neither $F_u$ nor $F_v$ is the attractor of any standard IFS with or without separation conditions, overlapping or otherwise. In fact we can also apply Theorem \ref{P2Q} or Theorem \ref{P2T} using the reflecting version of Condition (3), which means that neither $F_u$ nor $F_v$ is the attractor of any standard IFS, where the standard IFS may have defining similarities which reflect. 

\section{$n$-vertex IFSs}\label{five}

Before presenting the proof of Theorem \ref{P2Q} in Subsection \ref{five4}, we first discuss briefly Conditions (1), (2) and (3). 

\subsection{Theorem \ref{P2Q} - Condition (1)}\label{five1}

The next lemma shows that Condition (1) is in fact necessary for Theorem \ref{P2Q} and does not impose any effective restriction on the allowable directed graphs. In fact Lemma \ref{P2nv1} applies to any type of directed graph IFS as its proof is purely graph-theoretic. If  Condition (1) doesn't hold then all the simple cycles in the directed graph must pass through $u$. 
\begin{lem}
\label{P2nv1}
Let $\bigl( V, E^*, i, t, r, ((\mathbb{R},\left| \ \  \right|))_{v \in V}, (S_e)_{e \in E^1} \bigr)$ be an $n$-vertex $(n\geq 2)$ IFS with attractor $(F_u)_{u \in V}$. Let $u\in V$ be fixed and suppose that all the simple cycles in the directed graph are attached to $u$.

Then $F_u$ is the attractor of a standard IFS defined on $\mathbb{R}$. 
\end{lem}

\begin{proof} We can iterate Equation \eqref{Invariance equation3} $n$ times to obtain
\begin{equation*}
 F_u =  \bigcup_{ \substack{\be \in E_u^n } }  S_{\be}(F_{t(\be)})  =  \bigcup_{ \substack{\be \in E_{uu}^n } }  S_{\be}(F_u)  \cup \bigcup_{ \substack{\be \in E_{uv}^n \\ v \in V \\ v \neq u } }  S_{\be}(F_{t(\be)}). 
\end{equation*} 
If the last union is empty then $F_u   =  \bigcup_{ \substack{\be \in E_{uu}^n } }  S_{\be}(F_u)$ and $F_u$ is a standard IFS attractor. So we assume the last union is non-empty and, enumerating its paths, we write $F_u$ as 
\begin{equation}
\label{F_u eq}
 F_u  = \bigcup_{ \substack{\be \in E_{uu}^n } }  S_{\be}(F_u)  \cup \bigcup_{i=1}^m  S_{\be_i}(F_{t(\be_i)}).
\end{equation}
A simple path, which is not a cycle, of length $n$ has a vertex list which contains exactly $n+1$ different vertices so none of the paths $\be_i$ can be simple. A path that is not simple must contain a simple cycle, which is attached to $u$,  and this implies that for a given path $\be_i$ as we travel along its vertex list from $u$ to $t(\be_i)$ we must revisit the vertex $u$ at least once. Writing $\be_i$ in terms of its edges as $\be_i = e_{i,1}e_{i,2}\cdots e_{i,n}$ it follows that we may put $\be_i = \bc_{i,u}\Bf_i$ where $\bc_{i,u}= e_{i,1}e_{i,2}\cdots e_{i,j}$ is a simple cycle attached to $u$ and $\Bf_i = e_{i,j+1}e_{i,j+2}\cdots e_{i,n}$ is a path from $u$ to $t(\be_i)$. A cycle is normally independent of its initial and terminal vertices but here the $u$ in $\bc_{i,u}$ signifies that we only allow $i(\bc_{i,u}) = i(e_{i,1}) = u = t(e_{i,j}) = t(\bc_{i,u})$. 

Now $S_{\be_i}(F_{t(\be_i)}) \subset S_{\bc_{i,u}}(F_{t(\bc_{i,u})}) \subset F_u$ which means we can replace $S_{\be_i}(F_{t(\be_i)})$ by $S_{\bc_{i,u}}(F_{t(\bc_{i,u})})$ in Equation \eqref{F_u eq} and write $F_u$ as
\begin{equation*}
 F_u  = \bigcup_{ \substack{\be \in E_{uu}^n } }  S_{\be}(F_u)  \cup \bigcup_{i=1}^m  S_{\bc_{i,u}}(F_{t(\bc_{i,u})}) = \bigcup_{ \substack{\be \in E_{uu}^n } }  S_{\be}(F_u)  \cup \bigcup_{i=1}^m  S_{\bc_{i,u}}(F_u).
\end{equation*}
Therefore $F_u$ is a standard IFS attractor.
\end{proof}
 
For a concrete example which illustrates Lemma \ref{P2nv1} consider the $2$-vertex IFS of Figure \ref{P22VertexUnitInterval}. As stated in Section \ref{two} the minimum requirements for any directed graph are that it is strongly connected with at least two edges leaving each vertex. It follows that an $n$-vertex IFS will have a directed graph which contains at least $2n$ edges. Suppose we modify the graph of Figure \ref{P22VertexUnitInterval} by deleting the loops $e_1$ and $e_3$, replacing them with an edge from $u$ to $v$ and one from $v$ to $u$ so that the directed graph now contains the required $4$ edges with two edges still leaving each vertex. All the simple cycles in the directed graph are now of length $2$ and pass through both $u$ and $v$. Iterating Equation \eqref{Invariance equation3} twice we obtain $F_u =  \bigcup_{ \substack{\be \in E_u^2 } }  S_{\be}(F_{t(\be)})  =  \bigcup_{ \substack{\be \in E_u^2 } }  S_{\be}(F_u) $ so that $F_u$ is a standard IFS attractor and similarly for $F_v$.

\begin{figure}[htb]
\begin{center}
\includegraphics[trim = 5mm 130mm 3mm 15mm, clip, scale =0.7]{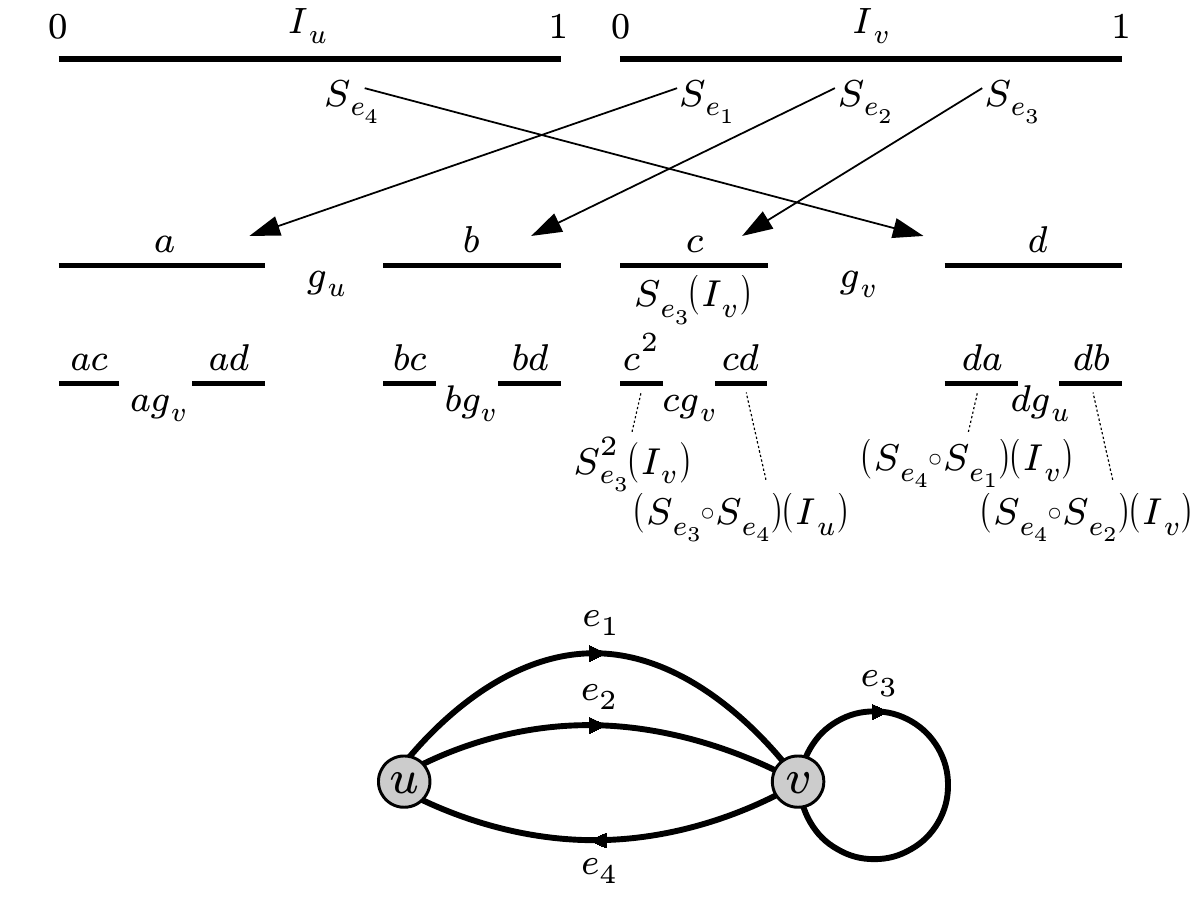}
\end{center}
\caption{A $2$-vertex IFS with just one loop in its directed graph.}
\label{P2VertexOneLoop}
\end{figure}

Now consider what happens if we retain just one loop, $e_3$, in modifying the graph of the $2$-vertex IFS of Figure \ref{P22VertexUnitInterval}. This situation is illustrated in Figure \ref{P2VertexOneLoop} where all the simple cycles are attached to $v$. Note that there is just one path, $e_3e_4$, of length $2$ from $v$ to $u$, that is $E_{vu}^2 = \left\{e_3e_4\right\}$, and $(S_{e_3}\circ S_{e_4})(F_u)\subset S_{e_3}(F_v) $. Following the proof of Lemma \ref{P2nv1} we can write $F_v$ as
\begin{align*}
F_v&= \bigcup_{ \substack{\be \in E_v^2 } }  S_{\be}(F_{t(\be)}) = \bigcup_{ \substack{\be \in E_{vv}^2 } }  S_{\be}(F_v)  \cup \bigcup_{ \substack{\be \in E_{vu}^2 } }  S_{\be}(F_u)  \\
&= S_{e_3}^2(F_v)  \cup (S_{e_4}\circ S_{e_1})(F_v) \cup (S_{e_4}\circ S_{e_2})(F_v) \cup (S_{e_3}\circ S_{e_4})(F_u) \\
&= S_{e_3}^2(F_v)  \cup (S_{e_4}\circ S_{e_1})(F_v) \cup (S_{e_4}\circ S_{e_2})(F_v) \cup S_{e_3}(F_v)  \\
&= (S_{e_4}\circ S_{e_1})(F_v) \cup (S_{e_4}\circ S_{e_2})(F_v) \cup S_{e_3}(F_v)
\end{align*}
where the last line follows since $S_{e_3}^2(F_v)\subset  S_{e_3}(F_v)$. This shows that $F_v$ is the attractor of a standard (CSSC) IFS. 

If Conditions (2) and (3) of Theorem \ref{P2Q} hold for $F_u$ then $F_u$ is not a standard IFS attractor. In fact it may also be possible to construct a proof, along the lines of the proofs of Lemmas \ref{P2O}, \ref{P2P} and Theorem \ref{P2M}, to prove that $F_u$ is not the attractor of any standard IFS provided only that $F_u \neq F_v$. 

This example, and the example illustrated in Figure \ref{P2VertexSubset} below, show that in certain cases $n$-vertex IFS attractors may have components which are a mix of standard IFS attractors and non-standard IFS components. For a complicated $n$-vertex (OSC) IFS if there is a component which is a standard (OSC) IFS attractor then this leads to an easier way of calculating the Hausdorff dimension, see Theorem \ref{dimension}. 

\subsection{Theorem \ref{P2Q} - Condition (2)}\label{five2}

The directed graph is strongly connected so there is always at least one simple path $\bp \in E^*_{uw}$ from the vertex $u$ to $w$. If Condition (2) holds then in particular it holds at the vertex $u$, since $u$ is in the vertex list of $\bp$, so that 
\begin{equation*}
\max G_u \leq \min G_u^1 \leq \max G_u^1 \leq \max G_u.
\end{equation*} 
This means all the level-$1$ gap lengths at the vertex $u$ must be equal with
\begin{equation}
\label{G_uG_u^1}
G_u^1 = \left\{g_u\right\} \text{ and }\max G_u =g_u
\end{equation} 
for some $g_u \in (0,1)$. 

As an example, for the $2$-vertex IFS of Figure \ref{P22VertexUnitInterval}, $G_u^1 = \left\{g_u\right\}$, $G_v^1 = \left\{g_v\right\}$ and by Lemma \ref{P2I}(a), $ \max G_u = \max\left\{g_u, bg_v\right\}$ so Condition (2) holds for $F_u$ if 
\begin{equation*}
bg_v \leq g_u \leq g_v.
\end{equation*} 
For the $2$-vertex IFS of Figure \ref{P2VertexOneLoop}, Condition (2) holds for $F_u$ if 
\begin{equation*}
\max\left\{ag_v, bg_v\right\} \leq g_u \leq g_v.
\end{equation*} 
For another example see the application of Theorem \ref{P2Q} in Subsection \ref{five3}. In general Condition (2) needs to be checked on a case by case basis. This is entirely possible as a constructive algorithm for calculating  sets of gap lengths as a finite union of cosets of finitely generated semigroups for any $n$-vertex (CSSC) IFS defined on the unit interval is given in \cite[Proposition 2.3.4]{phdthesis_Boore}. This means that $\max G_u$ can always be determined.

Condition (2) prevents similarity maps of components of attractors from spanning the gaps between level-$1$ intervals as we prove in the next Lemma.

\begin{lem}
\label{P2nospan}
Let $\bigl( V, E^*, i, t, r, ((\mathbb{R},\left| \ \  \right|))_{v \in V}, (S_e)_{e \in E^1} \bigr)$ be an $n$-vertex ($n\geq 2$, CSSC) IFS, defined on the unit interval, with attractor $(F_u)_{u \in V}$. Let $S: \mathbb{R} \rightarrow \mathbb{R}$ be a contracting similarity with contracting similarity ratio $r$, $0<r<1$, such that $S(F_u) \subset F_v$ for some, not necessarily distinct, $u,v \in V$. Suppose that $\max G_u \leq \min G_v^1$. 

Then $S(I_u)$ is contained in a level-$1$ interval of $F_v$, that is $S(I_u)\subset S(I_{t(e)})$ for some $e \in E_v^1$.
\end{lem}
\begin{proof}
First suppose $v=u$ with $S(F_u) \subset F_u$ and $\max G_u \leq \min G_u^1$. By \eqref{G_uG_u^1}, $\max G(S(F_u))=r\max G_u= rg_u < g_u= \min G_u^1$, which is enough to ensure that $S(I_u)$ doesn't span any gap between any level-$1$ intervals of $F_u$.

Secondly if $v\neq u$ then $\max G(S(F_u))=r\max G_u< \max G_u \leq \min G_v^1$ which again is enough to ensure that $S(I_u)$ doesn't span any gap between any level-$1$ intervals of $F_v$.
\end{proof}

\subsection{Theorem \ref{P2Q} - Condition (3)}\label{five3}
We state the next lemma for the purposes of discussion, omitting the proof.
\begin{lem}
\label{P2R}
Let $\bigl( V, E^*, i, t, r, ((\mathbb{R},\left| \ \  \right|))_{v \in V}, (S_e)_{e \in E^1} \bigr)$ be an $n$-vertex $(n\geq 2)$ IFS and suppose that its attractor $( F_u)_{u \in  V}$ has exactly $m$ distinct components with $1\leq m < n$.

Then an $m$-vertex IFS, $\bigl( V^\prime, E^{*\prime}, i^\prime, t^\prime, r^\prime, ((\mathbb{R},\left| \ \  \right|))_{v \in V^\prime}, (S^\prime_e)_{e \in E^{1\prime}} \bigr)$, can be constructed which has an attractor  $(F^\prime_u)_{u \in V^\prime}$ that consists of exactly the $m$ distinct components of  $( F_u)_{u \in  V}$. 
\end{lem}

\begin{figure}[htb]
\begin{center}
\includegraphics[trim = 10mm 160mm 10mm 15mm, clip, scale =0.7]{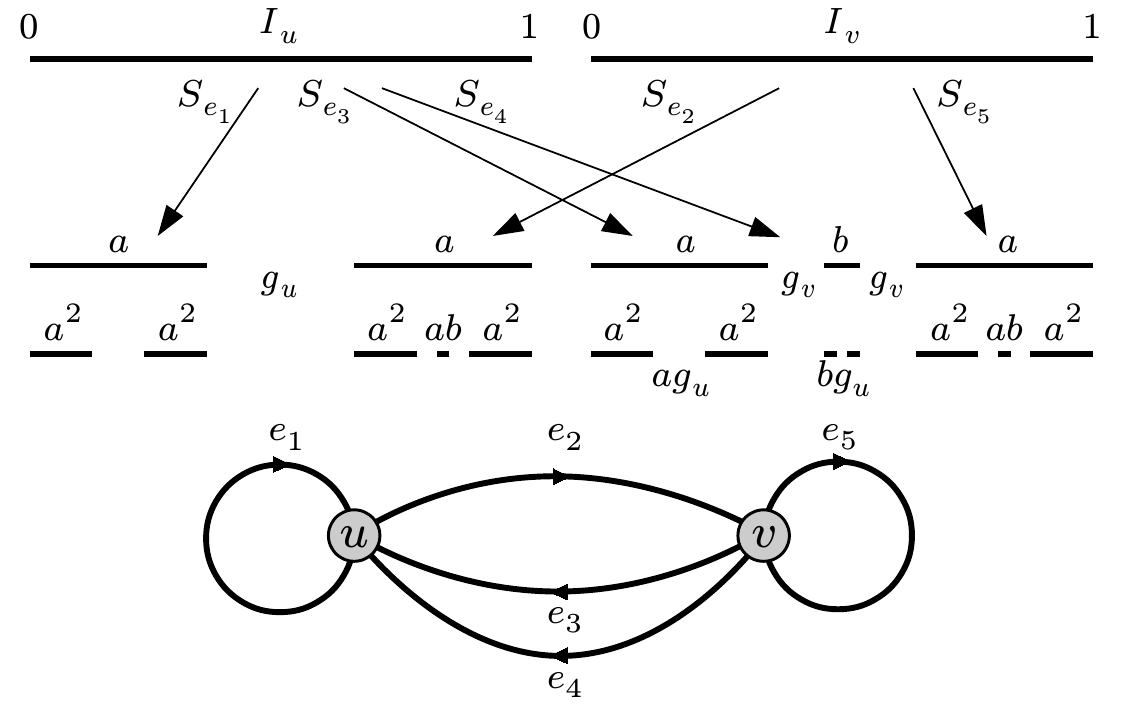}
\end{center}
\caption{A $2$-vertex IFS with attractor $(F_u,F_v)$ such that $F_u \neq F_v$, $F_u\subset F_v$ where $F_u$ is a standard IFS attractor.}
\label{P2VertexSubset}
\end{figure}
Lemma \ref{P2R} means that $F_u \neq F_v$ for all $u,v \in V$, $u \neq v$, would seem to be a natural generalisation of the condition of Theorem \ref{P2M}. However we can't do without $F_u\not\subset F_v$ in Condition (3) as it is needed in the proof of Subsection \ref{five4} . Clearly $F_u\not\subset F_v$ is stronger than $F_u \neq F_v$ since $F_u\not\subset F_v$ implies $F_u \neq F_v$ but the converse isn't true. The $2$-vertex (CSSC) IFS illustrated in Figure \ref{P2VertexSubset} provides an example for which Condition (1) of Theorem \ref{P2Q} holds and $F_u \neq F_v$,  but where $F_u\subset F_v$ and $F_u$ is a standard IFS attractor. Formally we can prove this as follows. From \eqref{Invariance equation3}
\begin{gather*}
F_u=S_{e_1}(F_u) \cup S_{e_2}(F_v) \\
F_v=S_{e_3}(F_u) \cup S_{e_4}(F_u) \cup S_{e_5}(F_v).
\end{gather*} 
Since $S_{e_1}=S_{e_3}$ and $S_{e_2}=S_{e_5}$ we obtain
\begin{equation*}
F_v=S_{e_1}(F_u) \cup S_{e_4}(F_u) \cup S_{e_2}(F_v)=F_u \cup S_{e_4}(F_u),
\end{equation*} 
which proves $F_u \neq F_v$ and $F_u \subset F_v$. It follows that
\begin{gather*}
F_u=S_{e_1}(F_u) \cup S_{e_2}(F_v) =S_{e_1}(F_u) \cup S_{e_2}(F_u \cup S_{e_4}(F_u))\\ 
=S_{e_1}(F_u) \cup S_{e_2}(F_u)\cup  (S_{e_2}\circ S_{e_4})(F_u)
\end{gather*} 
which shows $F_u$ is a standard IFS attractor. In fact the SSC holds for $F_u$ and not the CSSC. This simple example shows that the correct generalisation of the condition $F_u \neq F_v$ of Theorem \ref{P2M} is the stronger requirement that $F_u\not\subset F_v$ and $F_v\not\subset F_u$.  

Now suppose we change the edge $e_4$ in the directed graph of Figure \ref{P2VertexSubset} so that it becomes a loop at the vertex $v$. In this case $F_v = F_u \cup S_{e_4}(F_v)$ so that we still have $F_u \neq F_v$ and $F_u \subset F_v$ but now we have no way of deciding whether or not $F_u$ is a standard IFS attractor. New insights are needed for special cases like this.  

We end this subsection with an application of Theorem \ref{P2Q}. The notation used in Theorem \ref{P2Q} makes it appear more complicated than it actually is, so now we use the example in Figure \ref{P2VertexSubset} to show that it is easy to use in practice by applying it to $F_v$. We can put $\bc_u= e_1$ for Condition (1) and $\bp=e_3 \in E_{vu}^*$ with $V^\prime = \left\{u,v\right\}$. Clearly from Figure \ref{P2VertexSubset}, $\max G_v = \max \left\{g_v, ag_u, bg_u\right\}$ and Condition (2) requires 
\begin{align*}
\max G_v = \max \left\{g_v, ag_u, bg_u\right\} &\leq g_v = \min G_v^1, \\
\max G_v = \max \left\{g_v, ag_u, bg_u\right\} &\leq g_u = \min G_u^1. 
\end{align*}
It is always the case that $g_v < g_u$ so we only need $\max \left\{ag_u, bg_u\right\} \leq g_v $ to ensure Condition (2) holds. As examples $a=g_u=1/3$, $b=g_v=1/9$ will do, as will $a=b=1/4$, $g_u=1/2$, $g_v=1/8$. Clearly $F_v \not\subset F_u$ so that Condition (3) holds. Therefore if $\max \left\{ag_u, bg_u\right\} \leq g_v $ then $F_v$ is not a standard IFS attractor.

If $\max \left\{ag_u, bg_u\right\} > g_v $ is $F_v$ a standard IFS attractor? We discuss how further progress may be made to weaken Condition (2) and provide an answer in Section \ref{six}. 

\subsection{Proof of Theorem \ref{P2Q}}\label{five4}
The $n$-vertex directed graph IFS of  Theorem \ref{P2Q} is defined on the unit interval so that for each component of the attractor  $(F_u)_{u \in V}$, $\left\{0,1\right\}\subset F_u \subset I_u = [0,1]$. The proof uses the same ideas as the proof of Theorem \ref{P2M} in Subsection \ref{three1} but because it deals with general paths and cycles we present a formal proof by induction.  

\begin{proof}
Let $\be\in E_u^\mathbb{N}$ be the infinite path in the graph which starts at the vertex $u$, travels along the edges of the simple path $\bp$ until it reaches the vertex $w$ and then travels round the edges of the simple cycle $\bc_w$ indefinitely. The important thing about the path $\be$ is that the vertex $u$ appears just once in its vertex list as the initial vertex and doesn't appear thereafter. We write $\be$ as $\be=e_1e_2e_3\cdots $ with $i(\be)=i(e_1)=u$ and $t(e_i)\neq u$ for all $i \in \mathbb{N}$.

For a contradiction suppose 
\begin{equation}
\label{standardIFS}
 F_u = \bigcup_{j=1}^{n}  S_j(F_{u}),
\end{equation}
for some $n \geq 2$ so that $F_u$ is a standard IFS attractor. We now use induction to construct an infinite sequence $(n_k)$, with $n_k \in \mathbb{N}$ and $1 < n_{k+1} < n_k$ for each $k \in \mathbb{N}$, which is the required contradiction.

For $k \in \mathbb{N}$ let $P(k)$ be the following statement.

\medskip

\emph{For each $i$, $1 \leq i \leq k$, we can write $ F_{t(e_i)}$ as} 
 \begin{equation*}
 F_{t(e_i)} = \bigcup_{j=1}^{n_i}  S_{i,j}(F_{u}),
\end{equation*}
\indent \emph{for some contracting similarities $S_{i,j}$, with $1<n_k < n_{k-1}<\cdots < n_2 < n_1$.}

\medskip

\emph{Induction base}

By Condition (2) and Lemma \ref{P2nospan}, each similarity $S_j$ in Equation \eqref{standardIFS} is such that $S_j(I_u)$ is contained in a level-$1$ interval of $F_u$. Since $S_{e_1}(I_{t(e_1)})$ is a level-$1$ interval of $F_u$ there is at least one $S_j(I_u)$ with $S_j(F_u)\subset S_j(I_u)\subset S_{e_1}(I_{t(e_1)})$. Relabelling all the similarities of this type we obtain $\bigcup_{j=1}^{n_1}  S_j^\prime(I_{u})\subset S_{e_1}(I_{t(e_1)})$ where $1 \leq n_1 < n$. It follows by Equation \eqref{Invariance equation3} and the CSSC that in fact
\begin{equation*}
\bigcup_{j=1}^{n_1}  S_j^\prime(F_{u}) = S_{e_1}(F_{t(e_1)}).
\end{equation*}
If $n_1=1$ then $S_1^\prime(F_u)=S_{e_1}(F_{t(e_1)})$ for some single similarity $S_1^\prime$ which implies $S_1^\prime(0)=S_{e_1}(0)$, $S_1^\prime(1)=S_{e_1}(1)$,  so that $S_1^\prime=S_{e_1}$ and $F_u=F_{t(e_1)}$. However $u \neq {t(e_1)}$ so by Condition (3), $F_u \neq F_{t(e_1)}$. This means $1 < n_1 < n$. If for any $j$, $\left|S_j^\prime(F_u)\right|=\left|S_{e_1}(F_{t(e_1)})\right|$ then as  $S_j^\prime(F_u) \subset S_{e_1}(F_{t(e_1)})$ it follows that $S_j^\prime(0)=S_{e_1}(0)$, $S_j^\prime(1)=S_{e_1}(1)$, so that $S_j^\prime=S_{e_1}$ and $F_u \subset F_{t(e_1)}$. Again because $u \neq t(e_1)$ this is impossible by Condition (3). Therefore $\left|S_j^\prime(F_u)\right|<\left|S_{e_1}(F_{t(e_1)})\right|$ and the maps $S_{e_1}^{-1}\circ S_j^\prime$ are contracting similarities. Putting $S_{1,j}=S_{e_1}^{-1}\circ S_j^\prime$ we can write $F_{t(e_1)}$ as  
\begin{equation*}
F_{t(e_1)} = \bigcup_{j=1}^{n_1}  S_{1,j}(F_{u}),
\end{equation*}
for some contracting similarities $S_{1,j}$, with $1 < n_1$. This proves P(1).

\medskip

\emph{Induction hypothesis}

For $k \in \mathbb{N}$ we assume $P(k)$ is true. 

\medskip

\emph{Induction step}

From the induction hypothesis we can write $ F_{t(e_k)}$ as 
 \begin{equation}
\label{induction step}
 F_{t(e_k)} = \bigcup_{j=1}^{n_k}  S_{k,j}(F_{u}),
\end{equation}
for some similarities $S_{k,j}$. We now use exactly the same argument as in the induction base omitting some of the details. 

By Condition (2) and Lemma \ref{P2nospan}, each similarity $S_{k,j}$ in Equation \eqref{induction step} is such that $S_{k,j}(I_u)$ is contained in a level-$1$ interval of $ F_{t(e_k)}$. Since $S_{e_{k+1}}(I_{t(e_{k+1})})$ is a level-$1$ interval of $ F_{t(e_k)}$ there is at least one $S_{k,j}(I_u)$ with $S_{k,j}(F_u)\subset S_{k,j}(I_u)\subset S_{e_{k+1}}(I_{t(e_{k+1})})$. Relabelling all the similarities of this type we obtain $\bigcup_{j=1}^{n_{k+1}}  S_{k,j}^\prime(I_{u})\subset  S_{e_{k+1}}(I_{t(e_{k+1})})$ where $1 \leq n_{k+1} < n_k$. It follows by Equation \eqref{Invariance equation3} and the CSSC that in fact
\begin{equation*}
\bigcup_{j=1}^{n_{k+1}} S_{k,j}^\prime(F_{u}) = S_{e_{k+1}}(F_{t(e_{k+1})}).
\end{equation*}
If $n_{k+1}=1$ then $S_{k,1}^\prime(F_u)=S_{e_{k+1}}(F_{t(e_{k+1})})$ for some single similarity $S_{k,1}^\prime$  so that $S_{k,1}^\prime=S_{e_{k+1}}$ and $F_u=F_{t(e_{k+1})}$. However $u \neq t(e_{k+1})$ and by Condition (3) $F_u \neq F_{t(e_{k+1})}$. This means $1 < n_{k+1} < n_k$. If for any $j$, $\left| S_{k,j}^\prime(F_u)\right|=\left|S_{e_{k+1}}(F_{t(e_{k+1})})\right|$ then as  $ S_{k,j}^\prime(F_u) \subset S_{e_{k+1}}(F_{t(e_{k+1})})$ it follows that $ S_{k,j}^\prime=S_{e_{k+1}}$ and $F_u \subset F_{t(e_{k+1})}$. Again because $u \neq t(e_{k+1})$ this is impossible by Condition (3). Therefore $\left| S_{k,j}^\prime(F_u)\right|<\left|S_{e_{k+1}}(F_{t(e_{k+1})})\right|$ and the maps $S_{e_{k+1}}^{-1}\circ  S_{k,j}^\prime$ are contracting similarities. Putting $S_{k+1,j}=S_{e_{k+1}}^{-1}\circ  S_{k,j}^\prime$ we can write $F_{t(e_{k+1})}$ as   
\begin{equation*}
F_{t(e_{k+1})} = \bigcup_{j=1}^{n_{k+1}}  S_{k+1,j}(F_{u}),
\end{equation*}
for some contracting similarities $S_{k+1,j}$, with $1 < n_{k+1} < n_k$. Using the induction hypothesis, this means that for each $i$, $1 \leq i \leq k+1$, we can write $ F_{t(e_i)}$ as 
 \begin{equation*}
 F_{t(e_i)} = \bigcup_{j=1}^{n_i}  S_{i,j}(F_{u}),
\end{equation*}
for some similarities $S_{i,j}$, with $1<n_{k+1}<n_k < n_{k-1}<\cdots < n_2 < n_1$.  This shows that $P(k)$ implies $P(k+1)$ and completes the induction step.

\medskip

Therefore $P(k)$ is true for all $k \in \mathbb{N}$ and we have constructed an infinite sequence $(n_k)$, with $n_k \in \mathbb{N}$ and $1 < n_{k+1} < n_k$ for each $k \in \mathbb{N}$. This contradiction completes the proof of Theorem \ref{P2Q}.
\end{proof}

\subsection{Proof of Theorem \ref{P2T}}\label{five5}
The proof of Theorem \ref{P2T} relies on extending some of Feng and Wang's results for standard IFSs to $n$-vertex IFSs, see \cite{Paper_FengWang}.  We need the following two lemmas where we have omitted the proofs as they generalise in a straightforward way from \cite[Lemmas 3.5.1 and 3.5.7]{phdthesis_Boore}, see also \cite[Lemmas 7.1 and 7.3]{Paper_GCB_KJF} and  \cite[Theorem 4.1]{Paper_FengWang}. Lemma \ref{2GlemO} is only needed in a proof of Lemma \ref{2GlemT},  however we discuss it further in Section \ref{six}. For the calculation of the Hausdorff dimension, $s$, see Theorem \ref{dimension}.

\begin{lem}
\label{2GlemO}Let $\bigl( V, E^*, i, t, r, ((\mathbb{R},\left| \ \  \right|))_{v \in V}, (S_e)_{e \in E^1} \bigr)$ be an $n$-vertex ($n\geq 2$, OSC) IFS, defined on the unit interval, with attractor $(F_u)_{u \in V}$ and $s=\dimH F_u$. Let $S: \mathbb{R} \rightarrow \mathbb{R}$ be a contracting similarity such that $S(F_u) \subset F_v$ for some, not necessarily distinct, $u,v \in V$, and let $S(I_u)=\left[a_S,b_S\right]$. Suppose that $\mathcal{H}^s(F_u)=1$.
Then
\begin{align*}
&\textup{(a)}  \quad \mathcal{H}^s(S(F_u))=\mathcal{H}^s(F_v\cap S(I_u))=\left(b_S-a_S\right)^s, \\
&\textup{(b)}  \quad S(F_u)= F_v\cap S(I_u). 
\end{align*}   

\end{lem}

\begin{lem}
\label{2GlemT}
Let $\bigl( V, E^*, i, t, r, ((\mathbb{R},\left| \ \  \right|))_{v \in V}, (S_e)_{e \in E^1} \bigr)$ be an $n$-vertex ($n\geq 2$, OSC) IFS, defined on the unit interval, with attractor $(F_u)_{u \in V}$ and $s=\dimH F_u$. Let $S:\mathbb{R} \to \mathbb{R}$ and $T:\mathbb{R} \to \mathbb{R}$ be two distinct contracting similarities such that $S(F_u)\subset F_w$ and $T(F_v)\subset F_w$, for some, not necessarily distinct, $u,v,w \in V$. Suppose that $\mathcal{H}^s(F_u)=\mathcal{H}^s(F_v)=1$.

Then exactly one of the following three statements occurs
\begin{align*}
&\textup{(a)}  \quad S(I_u)\cap T(I_v)=\emptyset, \textrm{ which implies } S(F_u)\cap T(F_v)=\emptyset, \\
&\textup{(b)}  \quad S(I_u)\subset T(I_v), \textrm{ which implies } S(F_u)\subset T(F_v), \\
&\textup{(c)}  \quad T(I_v)\subset S(I_u), \textrm{ which implies } T(F_v)\subset S(F_u).
\end{align*}  
\end{lem}
To prove the next lemma we adapt part of the proof of \cite[Theorem 4.1]{Paper_FengWang} so that it applies to $n$-vertex IFSs. We use $\#A$ for the number of elements in a (finite) set $A$.
\begin{lem}
\label{2GlemU}
Let $\bigl( V, E^*, i, t, r, ((\mathbb{R},\left| \ \  \right|))_{v \in V}, (S_e)_{e \in E^1} \bigr)$ be an $n$-vertex ($n\geq 2$, CSSC) IFS, defined on the unit interval, with attractor $(F_u)_{u \in V}$ and $s=\dimH F_u$. Let $S: \mathbb{R} \rightarrow \mathbb{R}$ be a contracting similarity such that $S(F_u) \subset F_v$ for some, not necessarily distinct, $u,v \in V$. Suppose that the number of edges in the directed graph is minimal and  that $\mathcal{H}^s(F_u)=\mathcal{H}^s(F_{t(e)})=1$, for all $e \in E_v^1$. 

Then $S(I_u)$ is contained in a level-$1$ interval of $F_v$, that is $S(I_u)\subset S(I_{t(e)})$ for some $e \in E_v^1$. 
\end{lem}
\begin{proof}
As the number of edges in the associated directed graph, $\# E^1$,  is minimal, any other $\bigl( V, {E^*}^\prime, i^\prime, t^\prime, r^\prime, ((\mathbb{R},\left| \ \  \right|))_{v \in V}, ({S_e}^\prime)_{e \in {E^1}^\prime} \bigr)$ which has the same attractor  $(F_u)_{u \in V}$, will be such that $\# E^1\leq \#{E^1}^\prime$. Let $m$ be the number of level-$1$ intervals of $F_v$ that intersect $S(I_u)$, that is
\begin{equation*}
m =  \# \left\{ S_e(I_{t(e)}) : \ S(I_u) \cap S_e(I_{t(e)}) \neq \emptyset, \ e \in E_v^1 \right\}.
\end{equation*}
If $m\geq 2$ then by Lemma \ref{2GlemT} 
\begin{equation*}
S(F_u)= \bigcup_{ \substack{\ e \in E_v^1 \\ S(I_u) \cap S_e(I_{t(e)}) \neq \emptyset }}S_e(F_{t(e)}).
\end{equation*}
From Equation \eqref{Invariance equation3}
\begin{equation*}
 F_v = \bigcup_{ \substack{e\in E_v^1} } S_e(F_{t(e)}) = S(F_u) \cup \bigcup_{ \substack{e\in E_v^1 \\ S(I_u) \cap S_e(I_{t(e)}) = \emptyset} }S_e(F_{t(e)}) = \bigcup_{ \substack{e\in {E_v^1}^\prime} }S_e(F_{t(e)}),
\end{equation*} 
where ${E_v^1}^\prime$ is obtained from $E_v^1$ by replacing those edges $e$ for which $S(I_u) \cap S_e(I_{t(e)}) \neq \emptyset$ by a single edge $f$, from $v$ to $u$, with associated similarity $S_f = S$. Since $m\geq 2$, $\# {E_v^1}^\prime < \# E_v^1$ and we now have an $n$-vertex IFS with the same attractor $(F_u)_{u \in V}$ but where the number of edges in the directed graph is strictly less than $\#E^1$. This contradicts the minimality of $\#E^1$. Therefore $m=1$ and $S(I_u)\subset S(I_{t(e)})$ for some $e \in E_v^1$. 
\end{proof}
Lemma \ref{2GlemU} replaces the use of Lemma \ref{P2nospan} in the proof of Subsection \ref{five4} and so proves Theorem \ref{P2T}. 

\section{Conclusion}\label{six}

We conclude with a practical discussion of how Condition (2) of Theorem  \ref{P2Q} might be weakened. The proof of Theorem \ref{P2M} relies on Lemmas \ref{P2O} and \ref{P2P} but the arguments used in their proofs are not easy to adapt to general $n$-vertex ($n\geq 2$, CSSC) IFSs, so what follows is a consideration of other ways in which progress may be made. The sole purpose of Condition (2) is to prevent similarity maps of components of attractors from spanning the gaps between level-$1$ intervals as we showed in  Lemma \ref{P2nospan}. In fact most of the work of this paper has been in this direction (see also Lemmas \ref{P2O}, \ref{P2P}  and \ref{2GlemU}) because this is the main ingredient in the proofs of Subsections \ref{three1} and \ref{five4}.  

\begin{figure}%[htb]
\begin{center}
\includegraphics[trim = 9mm 10mm 10mm 15mm, clip, scale =0.74]{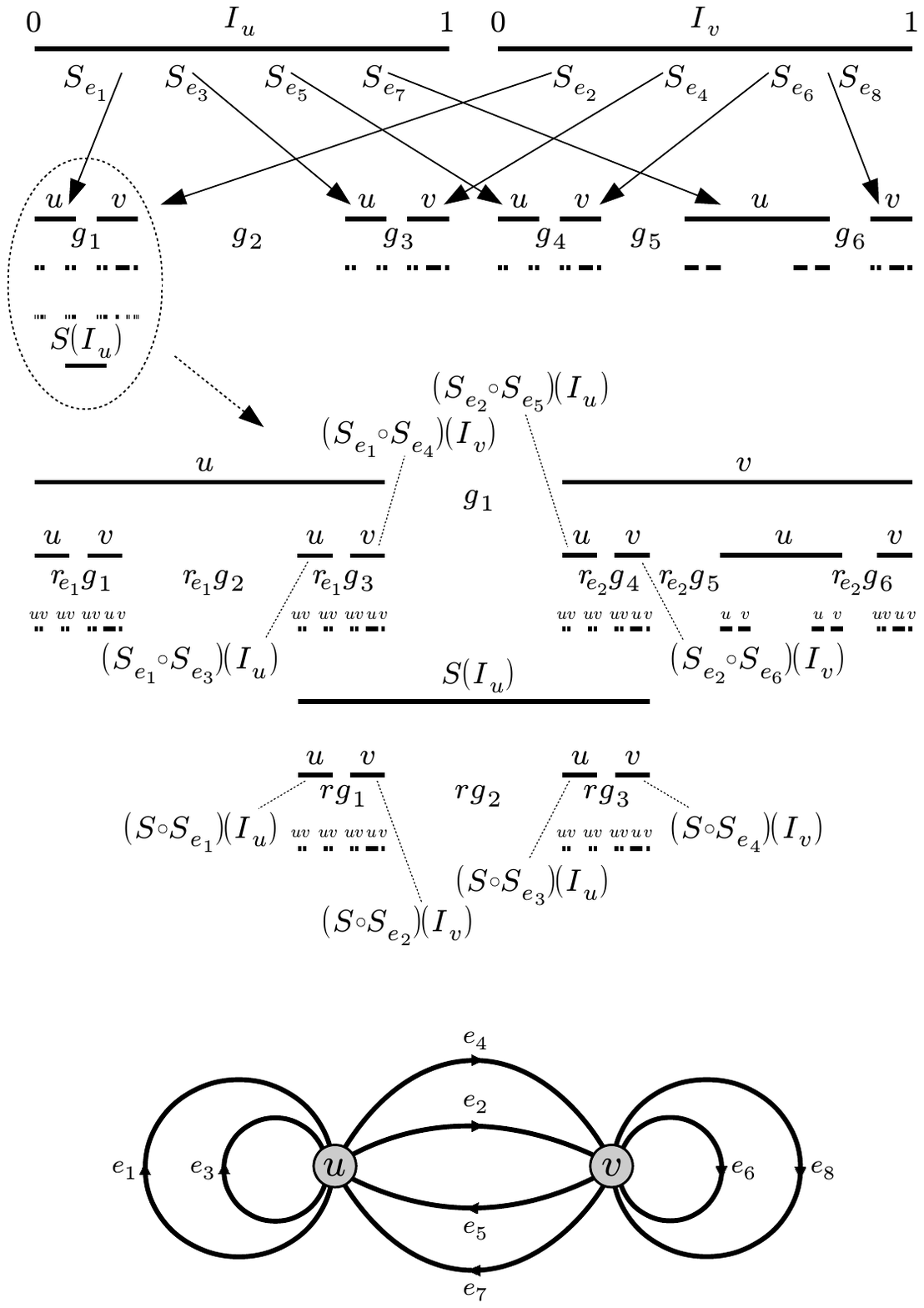}
\end{center}
\caption{A $2$-vertex IFS, defined on the unit interval, with attractor $(F_u,F_v)$, $F_u \not\subset F_v$ (and $F_v \not\subset F_u$) where $S$ is a contracting similarity such that $S(F_u)\subset F_u$ and $S(I_u)$ spans a gap between level-$1$ intervals}
\label{CounterExample}
\end{figure}

An interesting example of a standard (CSSC) IFS where a similarity map of the attractor is shown to span the gap between two level-$1$ intervals is given by Feng and Wang in \cite[Example 6.2]{Paper_FengWang}. This example would seem to be fairly special as it was produced independently in \cite[Theorem 6.2]{Paper_Elkes_Keleti_Mathe}. As illustrated in Figure \ref{CounterExample}, we can adapt this example to create a $2$-vertex IFS, defined on the unit interval, with attractor $(F_u,F_v)$, $F_u \not\subset F_v$ (and $F_v \not\subset F_u$) where $S$ is a contracting similarity such that $S(F_u)\subset F_u$ and $S(I_u)$ spans a gap between level-$1$ intervals. The contracting similarities are defined as
\begin{equation}
\label{S_{e_i}}
\begin{split}
S_{e_1}(x)=r_{e_1}x, \ S_{e_2}(x)=r_{e_2}x+r_{e_1}&+g_1, \ S_{e_3}(x)=r_{e_3}x+r_{e_1}+g_1+r_{e_2}+g_2, \\
S_{e_4}(x)=r_{e_4}x+1-r_{e_4}, \ S_{e_5}&(x)=r_{e_5}x, \ S_{e_6}(x)=r_{e_6}x+r_{e_5}+g_4, \\
S_{e_7}(x)=r_{e_7}x+r_{e_5}+g_4+&r_{e_6}+g_5, \ S_{e_8}(x)=r_{e_8}x+1-r_{e_8}, \\
S(x)=rx+r_{e_1}^2+&r_{e_1}g_1+r_{e_1}r_{e_2}+r_{e_1}g_2.
\end{split}
\end{equation}
The parameters in \eqref{S_{e_i}} are strictly positive and subject to the constraints
\begin{equation}
\label{constraints}
\begin{split}
r_{e_1}+g_1+r_{e_2}+g_2+r_{e_3}+g_3+r_{e_4}=1, \\
r_{e_5}+g_4+r_{e_6}+g_5+r_{e_7}+g_6+r_{e_8}=1.
\end{split}
\end{equation}
The level-$1$ intervals and gap lengths illustrated in Figure \ref{CounterExample} are drawn for a specific set of parameters which take the values 
\begin{equation}
\label{parameters}
\begin{split}
r=r_{e_8}=r_{e_i}=\frac{2}{20}, &\text{ for } 1\leq i \leq 6, \ r_{e_7}=\frac{7}{20}, \\ 
g_1=g_3=g_4=\frac{1}{20}, \ g_2&=\frac{10}{20}, \ g_5=\frac{4}{20} \text{ and } g_6=\frac{2}{20}. 
\end{split}
\end{equation}
Clearly it would be useful to know just how special this set of parameter values is. To this end we now consider the parameter values to be arbitrary. Assuming $S$ maps the four level-$1$ intervals of $F_u$ to four level-$2$ intervals of $F_u$, as shown in Figure \ref{CounterExample}, it follows that 
 \begin{gather*}
(S_{e_1}\circ S_{e_3})(I_u)=(S\circ S_{e_1})(I_u), \quad (S_{e_1}\circ S_{e_4})(I_v)=(S\circ S_{e_2})(I_v), \\
(S_{e_2}\circ S_{e_5})(I_u)=(S\circ S_{e_3})(I_u), \quad (S_{e_2}\circ S_{e_6})(I_v)=(S\circ S_{e_4})(I_v),
\end{gather*}
so that $r=r_{e_3}$ and 
 \begin{gather*}
r_{e_1}r_{e_4}=r_{e_2}r_{e_3}, \ r_{e_2}r_{e_5}=r_{e_3}^2, \ r_{e_2}r_{e_6}=r_{e_3}r_{e_4}, \\ 
g_1=r_{e_3}g_2, \ r_{e_1}g_3=r_{e_3}g_1, \ r_{e_2}g_4=r_{e_3}g_3. 
\end{gather*}
Here we have $6$ equations in $10$ unknowns which reduce to $6$ equations in $6$ unknowns if we fix the gap lengths $g_1,g_2,g_3,g_4$. This determines the following unique solution
\begin{equation}
\label{solution}
\begin{split}
r_{e_1}=\frac{g_1^2}{g_2g_3}, \ r_{e_2}=\frac{g_1g_3}{g_2g_4}, \ r_{e_3}=\frac{g_1}{g_2}, \ r_{e_4}=\frac{g_3^2}{g_2g_4}, \ r_{e_5}=\frac{g_1g_4}{g_2g_3}, \ r_{e_6}=\frac{g_3}{g_2}.
\end{split}
\end{equation}
Putting this solution into the constraint equations of \eqref{constraints} gives
\begin{gather}
\label{constraint1}
\frac{g_1^2}{g_2g_3}+g_1+\frac{g_1g_3}{g_2g_4}+g_2+\frac{g_1}{g_2}+g_3+\frac{g_3^2}{g_2g_4}=1, \\
\label{constraint2}
\frac{g_1g_4}{g_2g_3}+g_4+\frac{g_3}{g_2}+g_5+r_{e_7}+g_6+r_{e_8}=1.
\end{gather}
Any $2$-vertex (CSSC) IFS, with the similarities defined in \eqref{S_{e_i}} and with the associated directed graph as illustrated at the bottom of Figure \ref{CounterExample}, is obviously very special if its parameters satisfy the constraints of Equations \eqref{constraint1} and \eqref{constraint2}, thereby permitting $S(I_u)$ to span the gap between two level-$1$ intervals. To be clear about this suppose we choose  $g_1=g_3=g_4$, as is the case with the parameters of \eqref{parameters} that are used in Figure \ref{CounterExample}, then from \eqref{solution}, $r_{e_i}= g_1/g_2$ for $1 \leq i \leq 6$ and Equation \eqref{constraint1} reduces to a quadratic
\begin{equation*}
g_2^2 + (2g_1 - 1)g_2 + 4g_1=0
\end{equation*}
which has real solutions if  $g_1 \leq (20 - \sqrt{384})/8 \approx 0.0505$. This brief analysis shows that if we put $g_1=g_3=g_4=\alpha$, for $\alpha > (20 - \sqrt{384})/8$, then we can ensure that no such similarity $S$ can exist. 

We now describe another way in which the situation shown in Figure \ref{CounterExample} can be seen to be exceptional. We have labelled some level-$k$ intervals in Figure \ref{CounterExample}, $u$ or $v$ depending on whether an interval is a similarity map of $I_u$ or $I_v$ respectively. The symbols $u$ and $v$ for the level-$2$ and level-$3$ intervals immediately above $S(I_u)$ line up with the symbols $u$ and $v$ for the image level-$1$ and level-$2$ intervals immediately below $S(I_u)$. Matching up the symbols $u$ and $v$ places a very strong restriction on the possible directions for the similarities. To see this, suppose we were to swap the order of the level-$1$ intervals of $F_v$ so that the associated symbols read $vuuv$ instead of $uvuv$ (this can be done by changing the edge $e_5$ in the directed graph to a loop at $v$ and the loop $e_6$ at $v$ to an edge from $v$ to $u$). Then the level-$2$ intervals above $S(I_u)$ would now read $uvvu$ which doesn't match up with the level-$1$ intervals below $S(I_u)$ which read $uvuv$. In terms of similarities we would have $(S\circ S_{e_3})(I_u)=(S_{e_2}\circ S_{e_5})(I_v)$ which implies $F_u \subset F_v$ and $(S\circ S_{e_4})(I_v)=(S_{e_2}\circ S_{e_6})(I_u)$ which implies $F_v \subset F_u$ so that $F_u = F_v$ and the system reduces to a standard IFS which is excluded. In this case no such mapping $S$ can exist and there is the distinct advantage that we haven't needed to consider gap lengths or similarity ratios to prove it. 

Futher evidence that similarity maps like $S$ are special is provided by Elekes, Keleti and M\'{a}th\'{e} in \cite{Paper_Elkes_Keleti_Mathe}. They prove in 
 \cite[Lemma 4.8]{Paper_Elkes_Keleti_Mathe}, that for any given standard ($1$-vertex, self-similar, SSC) IFS, defined on $\mathbb{R}^m$, with attractor $F$, there are only a finite number of contracting similarities $S$, $S(F) \subset F$, such that $S(F)$ intersects at least two first generation elementary pieces. This means that for any standard (CSSC) IFS, defined on the unit interval, with attractor $F$, there are only a finite number of contracting similarities $S$, $S(F) \subset F$, such that $S(I)$ spans the gap between two level-$1$ intervals. It is reasonable to expect this result to carry over to any $n$-vertex ($n\geq 2$, CSSC) IFS, defined on the unit interval, so that there will only be at most a finite number, $N$, of similarities which span the gaps between level-$1$ intervals. The work of this paper suggests that in most cases we may expect $N=0$.

Taken together the above observations mean that we should be able to weaken Condition (2) of Theorem \ref{P2Q} and apply the proof of Subsection \ref{five4} to many other $n$-vertex IFSs, as long as the next conjecture is true.

\begin{conj}
\label{conjecture1}
Let $\bigl( V, E^*, i, t, r, ((\mathbb{R},\left| \ \  \right|))_{v \in V}, (S_e)_{e \in E^1} \bigr)$ be an $n$-vertex (CSSC) IFS, defined on the unit interval, with attractor $(F_u)_{u \in V}$. Let $S: \mathbb{R} \rightarrow \mathbb{R}$ be a contracting similarity such that $S(F_u) \subset F_v$ for some, not necessarily distinct, $u,v \in V$.

Then there exist $j,k \in \mathbb{N}$ with for each $\be \in E_u^j$ 
\begin{equation*}
(S\circ S_\be)(I_{t(\be)}) = S_\Bf(I_{t(\Bf)}),
\end{equation*}
for some $\Bf \in E_v^k$. In other words, $S$ maps the level-$j$ intervals of  $F_u$ to level-$k$ intervals of  $F_v$, for some $j,k \in \mathbb{N}$.  
\end{conj}
A proof of this conjecture would also provide a partial answer to \cite[Question 9.3]{Paper_Elkes_Keleti_Mathe} which asks the following related question. For an attractor $F$ of a standard (self-similar, $1$-vertex, SSC) IFS, defined on $\mathbb{R}^m$, and a similarity $S$ with $S(F)\subset F$, is  $S(F)$ always a finite union of elementary pieces of $F$? Or for $m=1$, is $S(F)$ always a finite union of level-$k$ elementary pieces for some $k\in \mathbb{N}$? 

If we could show that $S(F_u)= F_v\cap S(I_u)$ then this would certainly be of help in approaching a proof of Conjecture \ref{conjecture1}, as it should then be straight forward to show that end-points of level-$j$ intervals are mapped to end-points of level-$k$ intervals and also it would ensure that gap intervals are mapped to other gap intervals. It is certainly the case that $S(F_u)= F_v\cap S(I_u)$ under the OSC with the very strong condition that $\mathcal{H}^s(F_u)=1$, see Lemma \ref{2GlemO}, however this is not the case under the CSSC alone. To see this we can modify the example in Figure \ref{CounterExample} so that $S(F_u)\subsetneqq F_u\cap S(I_u)$. All we need to do is to add three similarities which map $I_u$ into the gaps between $S_{e_1}(I_u)$ and $S_{e_2}(I_v)$, $S_{e_3}(I_u)$ and $S_{e_4}(I_v)$, and $S_{e_5}(I_u)$ and $S_{e_6}(I_v)$. 

On the other hand if a counter-example to Conjecture \ref{conjecture1} can be constructed then new ideas will be needed to make further progress. 

Of course for the specific $2$-vertex example shown in Figure \ref{CounterExample}, with the parameters of \eqref{parameters} and attractor $(F_u,F_v)$ with $F_u \not\subset F_v$ (and $F_v \not\subset F_u$), where $S(I_u)$ does indeed span a gap between level-$1$ intervals, there remains the interesting question as to whether or not $F_u$ is a standard IFS attractor. 

New insights will be needed to answer questions like these.

%\bibliographystyle{amsplain}
 
%\bibliography{P}

\end{document}